%% file: main.tex
\documentclass[submission,copyright,creativecommons]{eptcs}
\providecommand{\event}{QPL 2023}

\usepackage[utf8]{inputenc}

\usepackage{amsmath,amsthm,amssymb,amsfonts,soul}
\theoremstyle{definition}
\newtheorem{theorem}{Theorem}[section]
\newtheorem*{theorem*}{Theorem}

\newtheorem{proposition}[theorem]{Proposition}

\newtheorem{defn}[theorem]{Definition}

\newtheorem{example*}[theorem]{Example*}
\newtheorem{examples*}[theorem]{Examples*}
\newtheorem{remark}[theorem]{Remark}
\newtheorem{remark*}[theorem]{Remark*}

\usepackage{tikz-cd}
\usepackage{tikz}
\usepackage{pgfplots}
\pgfplotsset{compat=1.17}
\usetikzlibrary{trees}
\usetikzlibrary{topaths}
\usetikzlibrary{decorations.pathmorphing}
\usetikzlibrary{fadings}
\usetikzlibrary{decorations.pathreplacing}
\usetikzlibrary{decorations.markings}
\usetikzlibrary{matrix,backgrounds,folding}
\usetikzlibrary{chains,scopes,positioning,fit}
\usetikzlibrary{arrows,shadows}
\usetikzlibrary{calc} 
\usetikzlibrary{chains}
\usetikzlibrary{shapes,shapes.geometric,shapes.misc}
\usetikzlibrary{circuits.ee.IEC}
\usetikzlibrary{decorations.markings}

\newcommand{\tikzfig}[1]{%
\IfFileExists{#1.tikz}
  {\input{#1.tikz}}
  {%
    \IfFileExists{./figures/#1.tikz}
      {\input{./figures/#1.tikz}}
      {\tikz[baseline=-0.5em]{\node[draw=red,font=\color{red},fill=red!10!white] {\textit{#1}};}}%
  }%
}

\pgfdeclarelayer{edgelayer}
\pgfdeclarelayer{nodelayer}
\pgfsetlayers{background,edgelayer,nodelayer,main}
\tikzstyle{none}=[inner sep=0mm]
\tikzstyle{every loop}=[]
\tikzstyle{mark coordinate}=[inner sep=0pt,outer sep=0pt,minimum size=3pt,fill=black,circle]

\tikzstyle{smalldotb}=[fill=black, inner sep=0mm,minimum width=1mm,minimum height=1mm,draw,shape=circle]
\tikzstyle{H}=[-, style=dashed]

\usepackage{color}
\def\bR{\begin{color}{red}}  
\def\bB{\begin{color}{blue}}
\def\bM{\begin{color}{magenta}}  
\def\bC{\begin{color}{cyan}}
\def\bW{\begin{color}{white}}
\def\bBl{\begin{color}{black}}
\def\bG{\begin{color}{green}}
\def\bY{\begin{color}{yellow}}
\def\e{\end{color}}

\title{Constructor Theory as Process Theory}
\author{
    Stefano Gogioso
    \institute{Hashberg Ltd}
    \institute{University of Oxford}
    \email{stefano.gogioso@cs.ox.ac.uk} 
    \and
    Vincent Wang-Ma\'{s}cianica
    \institute{Quantinuum Ltd}
    \institute{University of Oxford}
    \email{vincent.wang@cambridgequantum.com}
    \and
    Muhammad Hamza Waseem
    \institute{Quantinuum Ltd}
    \institute{University of Oxford}
    \email{hamza.waseem@physics.ox.ac.uk}. 
    \and
    Carlo Maria Scandolo
    \institute{University of Calgary}
    \email{carlomaria.scandolo@ucalgary.ca} 
    \and
    Bob Coecke
    \institute{Quantinuum Ltd}
    \email{bob.coecke@quantinuum.com}
}
\def\titlerunning{Constructor theory as Process Theory}
\def\authorrunning{S. Gogioso, V. Wang-Ma\'{s}cianica, M. H. Waseem, C. M. Scandolo and B. Coecke}

\newcommand{\copyrightholders}{S. Gogioso, V. Wang-Ma\'{s}cianica, \\M. H. Waseem, C. M. Scandolo and B. Coecke}

\begin{document}

\maketitle

\begin{abstract}
    Constructor theory is a meta-theoretic approach that seeks to characterise concrete theories of physics in terms of the (im)possibility to implement certain abstract ``tasks'' by means of physical processes. Process theory, on the other hand, pursues analogous characterisation goals in terms of the compositional structure of said processes, concretely presented through the lens of (symmetric monoidal) category theory. In this work, we show how to formulate fundamental notions of constructor theory within the canvas of process theory. Specifically, we exploit the functorial interplay between the symmetric monoidal structure of the category of sets and relations, where the abstract tasks live, and that of symmetric monoidal categories from physics, where concrete processes can be found to implement said tasks. Through this, we answer the question of how constructor theory relates to the broader body of process-theoretic literature, and provide the impetus for future collaborative work between the fields.
\end{abstract}

\section{Introduction}
Constructor theory \cite{deutsch2013constructor,deutsch2015constructor,marletto_science_2021} is a metatheoretic approach that seeks to characterise concrete theories of physics and information in terms of the \emph{possibility} and \emph{impossibility} of \emph{tasks}, which are \emph{transformations} between \emph{systems}.
Transformations may require auxiliary inputs other than the system to be transformed: the task of turning black shoes into white shoes may require a stock of white paint as an auxiliary input in addition to the black shoes themselves.
Tasks transform \emph{states} of systems into other states, and \emph{attributes} of systems---such as the blackness of a shoe---into other attributes.
In our universe, we will eventually run out of white paint for this task, but if we had a mathematically ideal paintbrush with infinite white paint, we could reuse it for as many instances of the task as we'd like; such non-exhaustible auxiliary catalysts for tasks are called \emph{constructors}.
A task is \emph{possible} when it is partnered with a constructor that allows the task to be performed arbitrarily many times, and the task is \emph{impossible} otherwise.
Though constructors and tasks are abstract, they provide explanatory value; constructor theory seeks to characterise physical theories in terms of what tasks are possible.
As a metatheory, constructor theory is implementation-agnostic, and one can choose whatever formal system of mathematics they like as a concrete language to interpret the italic terms above.

Process theories provide one such mathematically formal language, one particularly well-suited to describe the composition of processes in space-time.
Moreover, process theories are expressed in terms of string diagrams, which are an aesthetic, intuitive, flexible, and rigorous metalinguistic syntax, empowering the modeller by allowing them to operate at a level of abstraction of their choice.
This means that the same abstract diagrams provide a common syntactic foundation for fields as disparate as linear and affine algebra \cite{bonchi_interacting_2017,bonchi_graphical_2019}, first order logic \cite{haydon_compositional_2020}, electrical circuits \cite{boisseau_string_2022}, digital circuits \cite{ghica_categorical_2016}, database operations \cite{hefford_categories_2020,wilson_safari_2021}, spatial relations \cite{wang-mascianica_talking_2021}, game theory \cite{hedges_string_2015}, petri nets \cite{baez_open_2020}, hypergraphs \cite{alvarez-picallo_rewriting_2022}, probability theory \cite{cho_disintegration_2019,fritz_finettis_2021}, causal reasoning \cite{jacobs_causal_2019}, machine learning \cite{cruttwell_categorical_2022}, and quantum theory \cite{coecke_picturing_2017,coecke_quantum_2023}, to name just a few.

In this short paper, we provide a formal interpretation of constructor-theoretic terminology and ideas within the string-diagrammatic setting of process theories, with the intent to build a bridge between the two communities.
We caution against the view that constructor theory is ``just'' a class of process theories, in the same sense as it would be misguided to claim that prime numbers are ``just'' integers.
Process theory merely provides a rigorous mathematical language for constructor theorists to tell their stories.

For the process theorists in our audience, we wish to stress that the pedagogical mathematical presentation of this paper is for the sake of constructor theorists who might be approaching our field for the first time.
Regardless, we offer you a Rosetta stone for constructor theory within what we understand to be the \emph{de rigueur} mathematics of the field, transliterated into diagrams with as few embellishments and interpretational choices as possible.
For the constructor theorists in our audience, we extend a warm invitation to join the process-theoretic community: to the best of our knowledge, this is the most attractive and general formal arena available within which to explore the ramifications of constructor theory.

\newcommand{\textdef}[1]{\textbf{#1}}
\newcommand{\obj}[1]{\ensuremath{\text{obj}\!\left(#1\right)}}
\newcommand{\states}[2]{\ensuremath{\text{states}_{#1}\!\left(#2\right)}}
\newcommand{\id}[1]{\ensuremath{\text{id}_{#1}}}
\newcommand{\substr}[1]{\ensuremath{\texttt{#1}}}
\newcommand{\smc}[1]{\ensuremath{\textbf{#1}}}
\newcommand{\Set}{\smc{Set}}
\newcommand{\Rel}{\smc{Rel}}
\newcommand{\EndoRel}{\smc{EndoRel}}
\newcommand{\tmapsto}[1]{\ensuremath{\stackrel{\task{#1}}{\mapsto}}}
\newcommand{\suchthat}[2]{\ensuremath{\left\{#1\;\middle|\;#2\right\}}}
\newcommand{\task}[1]{\ensuremath{\mathfrak{#1}}}
\newcommand{\precond}[2]{\ensuremath{{#1}|_{#2}}}
\newcommand{\postcond}[2]{\ensuremath{{#1}|^{#2}}}
\newcommand{\ppcond}[3]{\ensuremath{{#1}|_{#2}^{#3}}}
\newcommand{\indtask}[1]{\ensuremath{\left\lfloor #1 \right\rfloor}}
\newcommand{\possibletasks}[1]{\ensuremath{#1^\checkmark}}
\newcommand{\quotask}[3]{#1|_{#2}^{#3}}

\section{Conceivable Tasks}

Constructor theory is concerned with the study of physical theories in terms of the question ``which \emph{tasks} are performable within this physical theory?": there is an abstract notion of \emph{conceivable} tasks and a concrete notion of \emph{possible} tasks.
In seminal work by Deutsch \cite{deutsch2013constructor}, it was remarked that, in full generality, the only real requirement on conceivable tasks is arbitrary composability in sequence and in parallel, i.e. that they form a symmetric monoidal category (SMC).
\footnote{It is possible that Deutsch meant for substrates to have an individual identity as physical systems, rather than just a ``type'': that is, it is possible that Deutsch would prefer for ``this qubit'' and ``that qubit'' to be modelled by different---albeit isomorphic---objects in a process theory. In this case, it would make no sense to consider parallel compositions of tasks involving the ``same'' physical system, and partially-monoidal categories as defined in \cite{gogioso2019church} would be preferable as a process-theoretical universe.}
Back then, however, the same author made a specific choice to model tasks as relations between sets: constructor theory literature has stuck by this choice ever since, and so will we.

\begin{remark}
    In this work, we take all monoidal categories to be \emph{strict}, and in particular we assume that objects $\obj{\smc{D}}$ in a monoidal category $\smc{D}$ form a strict monoid. In the case of the SMCs $\Rel$ and $\Set$, considered in Definition 2.2 below, this implies a choice of singleton set $1 := \{*\}$ to act as a strict unit for the Cartesian product:
    \[
        X \times 1 = X = 1 \times X
    \]
    This also affords us the freedom to write triples (and other tuples) without having to care about nesting:
    \[
        X \times Y \times Z
        = \suchthat{(x,y,z)}{x \in X, y \in Y, z \in Z}
    \]
    Note that strictness does not extend to symmetry isomorphisms: we have that $X \times Y \cong Y \times X$, but this doesn't mean that $X \times Y = Y \times X$. As a consequence, the monoid formed by objects in a strict SMC is not generally commutative.
\end{remark}

We take the theory of \textit{conceivable tasks} to be $\Rel$, the $\dagger$-SMC of sets and relations. We write $\task{A}: X \rightarrow Y$ for a task/relation $\task{A} \subseteq X \times Y$, where the set $X$ labels legitimate input states for the task and the set $Y$ labels legitimate output states. To help distinguish between pairs/tuples of elements in a Cartesian product and pairs of domain/codomain elements in a relation, we reserve pair/tuple notation for the former and adopt \emph{maplet} notation for the latter:
\[
    x \mapsto y \;\;:\equiv\;\; (x,y)
    \hspace{2.5cm}
    x \tmapsto{A} y \;\;:\equiv\;\; (x,y) \in \task{A}
\]
We omit $\task{A}$ from $\stackrel{\task{A}}{\mapsto}$ when clear from context. The \textit{sequential composition} $\task{B} \circ \task{A}: X \rightarrow Z$ of task $\task{B}: Y \rightarrow Z$ after $\task{A}: X \rightarrow Y$ is defined as follows:
\[
    \task{B} \circ \task{A}
    := \suchthat{x\mapsto z}{\exists y \in Y.\, x \stackrel{\task{A}}{\mapsto} y \text{ and } y \stackrel{\task{B}}{\mapsto} z}
\]
Sequential composition in diagrammatic language:
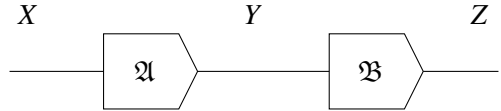
\[
    \input{seqcomp.tikz}
\]
Composite sets of states are obtained by Cartesian product $X \times Y$:
\[
    X \times Y
    := \suchthat{(x,y)}{x \in X \text{ and }y \in Y}
\]
The \textdef{parallel composition} $\task{A} \times \task{B}: X \times Z \rightarrow Y \times W$ of tasks $\task{A}: X \rightarrow Y$ and $\task{B}: Z \rightarrow W$ is defined as follows:
\[
    \task{A} \times \task{B}
    := \suchthat{ (x,z) \mapsto (y,w) }{ x \tmapsto{A} y \text{ and } z \tmapsto{B} w }
\]
Parallel composition in diagrammatic language:
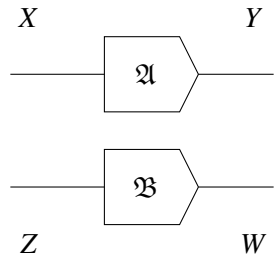
\[
    \input{parcomp.tikz}
\]
The \textdef{transpose} $\task{A}^\dagger: Y \rightarrow X$ of a task $\task{A}: X \rightarrow Y$ is defined as follows:
\[
    \task{A}^\dagger
    := \suchthat{ y \mapsto x}{x \tmapsto{A} y }
\]
Transposition in diagrammatic language:
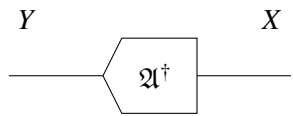
\[
    \input{transpose.tikz}
\]
Finally, there are \textdef{symmetry isomorphisms} (aka \textdef{swaps}) $\sigma_{X,Y}: X \times Y \stackrel{\cong}{\rightarrow} Y \times X$:
\[
    \sigma_{X,Y} := \suchthat{ (x,y) \mapsto (y,x) }{ x \in X \text{ and } y \in Y }
\]
Symmetry isomorphisms in diagrammatic language:
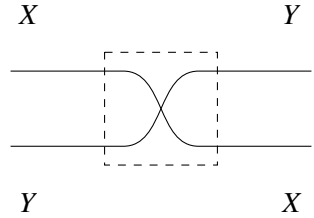
\[
    \input{syms.tikz}
\]
The symmetry isomorphisms are a structural feature of the category, making it possible to compose relations into acyclic networks, where outputs of relations can be connected to inputs of other relations. This is made possible by the following properties of the symmetry isomorphisms:

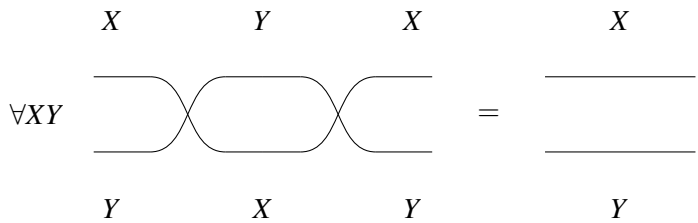
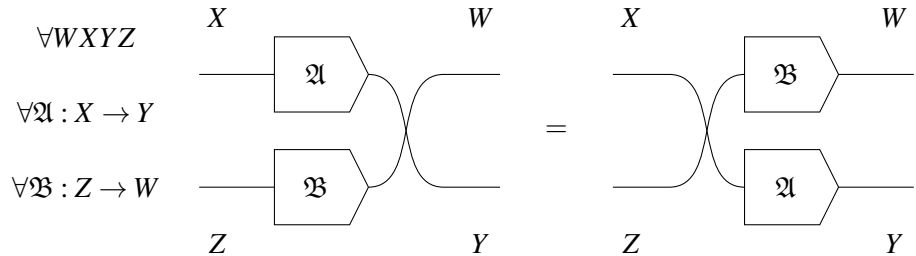
\[
    \input{symeqs.tikz}
\]

\begin{remark}
    The sets $X$ and $Y$ are allowed to be distinct, for sake of generality. Asking that they are always equal is equivalent to restricting the theory of conceivable tasks to be the $\dagger$-SMC $\EndoRel$ of sets and endo-relations $R: X \rightarrow X$, which is a sub-$\dagger$-SMC of $\Rel$.
\end{remark}

If we restrict our attention to the \emph{total deterministic} relations in $\Rel$, we obtain the sub-SMC $\Set$ of sets and functions between them. Functions are closed under acyclic network composition (sequential and parallel, including the usage of symmetry isomorphisms), but not under transpose. Important examples of functions are the \textdef{copy map} $\delta_X: X \rightarrow X \times X$ and \textdef{discarding map} $\epsilon_X: X \mapsto 1$ on a set $X$:
\begin{align*}
    \delta_X &:= \suchthat{x \mapsto (x,x)}{ x \in X}\\
    \epsilon_X &:= \suchthat{x \mapsto *}{ x \in X}
\end{align*}
Copy and delete maps in diagrammatic language:
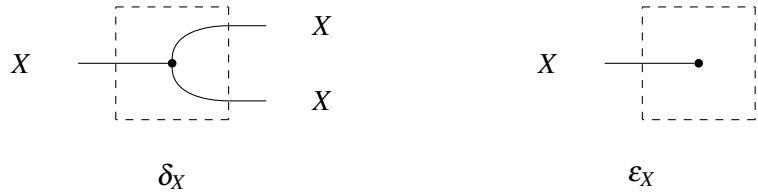
\[
    \input{copydel.tikz}
\]
Copies are indistinguishable under swaps and repeated copies, and deleting a copy results in the identity:
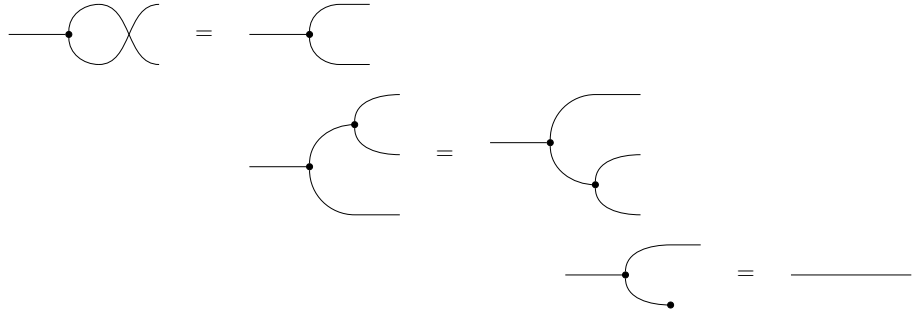
\[
    \scalebox{0.8}{$
    \input{copydeleqs.tikz}
    $}
\]
The transposes of the copy and discarding map are not functions. The transpose $\delta_X^\dagger: X \times X \rightarrow X$ is the \textdef{match map}, a partial function which returns the common value of its inputs when they're equal and is otherwise undefined:
\[
    \delta_X^\dagger := \suchthat{(x,x) \mapsto x}{ x \in X}
\]
Match map in diagrammatic language:
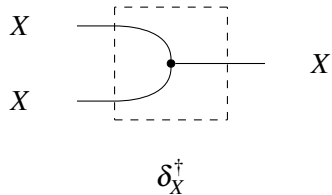
\[
    \input{match.tikz}
\]
Relations $S: 1 \rightarrow X$, such as the transpose $\epsilon_X^\dagger: 1 \rightarrow X$ of the discarding map, can be identified with all possible \textdef{attributes} of states in $X$, i.e. with all possible subsets $S \subseteq X$:
\[
    S \cong \suchthat{* \mapsto x}{ x \in S}
\]
States and attributes have the same notation in diagrammatic language, since states $x \in X$ can be identified with singleton subsets $\{x\} \subseteq X$:
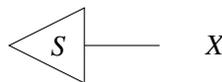
\[
    \input{attr.tikz}
\]
The transpose $\epsilon_X^\dagger: 1 \rightarrow X$ of the discarding map is the \textdef{trivial attribute}, corresponding to subset $X \subseteq X$:
\[
    \eta_X := \epsilon_X^\dagger = \suchthat{* \mapsto x}{ x \in X}
\]
Trivial attribute in diagrammatic language:
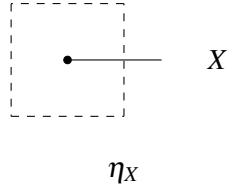
\[
    \input{unit.tikz}
\]
Attributes can be used to condition tasks to specific input states.
\begin{defn}
    Let $\task{A}: X \times Z \rightarrow Y$ be a task and let $S \subseteq Z$ be an attribute on states in $Z$. The \textdef{pre-conditioned task} is defined to be the task obtained by forgetting all information about the $Z$ input of $\task{A}$ other than the fact that the input state has attribute $S$:
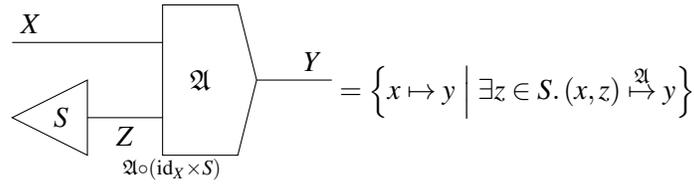
    \[    
        \substack{
            \input{preconditioned.tikz}\\
            \task{A} \circ (\id{X} \times S)
        }
        = \suchthat{x \mapsto y}{\exists z \in S.\, (x,z) \tmapsto{A} y}
    \]
    As a special case, we can discard the $Z$ input entirely, by pre-conditioning against the trivial attribute $\eta_Z$:
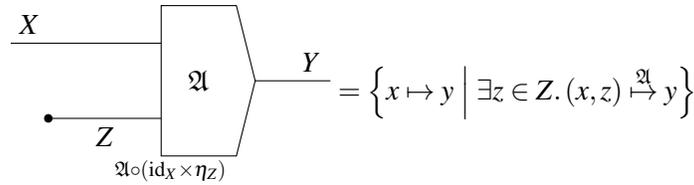
    \[
        \substack{
            \input{prediscard.tikz}\\
            \task{A} \circ (\id{X} \times \eta_Z)
        }
        = \suchthat{x \mapsto y}{\exists z \in Z.\, (x,z) \tmapsto{A} y}
    \]
\end{defn}
The object $1$ is \emph{terminal} in $\Set$: there is a unique function $\epsilon_X: X \rightarrow 1$ for any set $X$. However, it is not terminal in $\Rel$: the relations $X \rightarrow 1$ are exactly the transposes $S^\dagger: X \rightarrow 1$ of the attributes $S: 1 \rightarrow X$. Explicitly, they are the constant partial functions with the attribute $S$ as their domain:
\[
    S^\dagger := \suchthat{ x \mapsto * }{ x \in S }
\]
The transposes of attributes are \textdef{tests}, which can be used to condition tasks to specific output states.
\begin{defn}
    Let $\task{A}: X \rightarrow Y \times Z$ be a task and let $S \subseteq Z$ be an attribute on states in $Z$. The \textdef{post-conditioned task} is defined to be the task obtained by forgetting all information about the $Z$ output of $\task{A}$ other than the fact that the output state has attribute $S$:
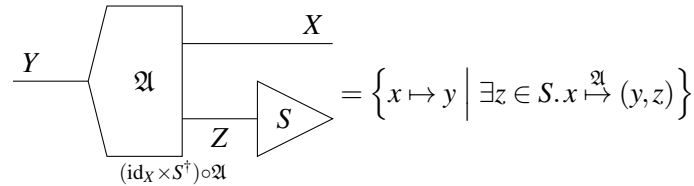
    \[
        \substack{
            \input{postconditioned.tikz}\\
            (\id{X} \times S^\dagger) \circ \task{A}
        }
        = \suchthat{x \mapsto y}{\exists z \in S.\, x \tmapsto{A} (y, z)}
    \]
    As a special case, we can discard the $Z$ output of the task entirely, by post-conditioning against the trivial attribute on $Z$:
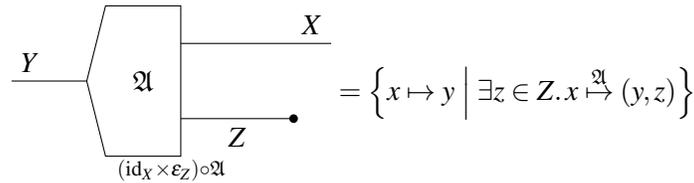
    \[
        \substack{
            \input{postdiscard.tikz}\\
            (\id{X} \times \epsilon_Z) \circ \task{A}
        }
        = \suchthat{x \mapsto y}{\exists z \in Z.\, x \tmapsto{A} (y, z)}
    \]
\end{defn}
\begin{remark}
    We can simultaneously pre-condition a task $\task{A}: X \times Z \rightarrow Y \times W$ against an attribute $P \subseteq Z$ and post-condition it against an attribute $Q \subseteq W$:
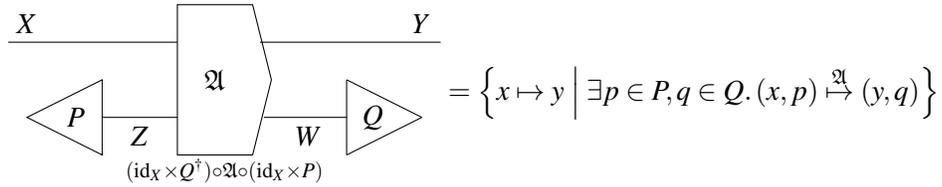
    \[
        \substack{
            \input{prepost.tikz}\\
            (\id{X} \times Q^\dagger) \circ \task{A} \circ (\id{X} \times P)
        }
        = \suchthat{x \mapsto y}{\exists p \in P, q \in Q.\, (x, p) \tmapsto{A} (y, q)}
    \]
\end{remark}

\section{Possible Tasks}

Conceivable tasks are a theory-independent concept: they provide a formal universe within which to formulate principles and derive constraints. On the other hand, possible tasks are theory-dependent, induced by the constructors physically available to implement them. In order to determine which tasks are possible, we need to make a choice of substrates within a theory of processes.

\begin{defn}
    A \textdef{choice of substrates} $\left( \smc{C}, \Sigma, \Gamma \right)$ comprises:
    \begin{enumerate}
        \item  A reference \textdef{theory of processes}, in the form of a strict SMC $\smc{C} = \left(\obj{\smc{C}}, \otimes, I\right)$. For example, this could be the theory of finite-dimensional quantum systems and unitary transformations.
        \item A choice of \textdef{substrates}, in the form of a subset $\Sigma \subseteq \obj{\smc{C}}$ of systems in the theory of processes. 
        \item A choice of \textdef{sets of substrate states}, in the form of a family $\Gamma = \left( \Gamma_{\substr{H}} \right)_{\substr{H} \in \Sigma}$ where $\Gamma_\substr{H} \subseteq \states{\smc{C}}{\substr{H}}$ is a set of states in $\smc{C}$ for each substrate $\substr{H} \in \Sigma$.
    \end{enumerate}
    We require that the choice of substrates be closed under parallel composition: $I \in \Sigma$ and $\substr{H} \otimes \substr{K} \in \Sigma$ for all $\substr{H}, \substr{K} \in \Sigma$. We further require that the set of substrate states respects parallel composition of substrates: $\Gamma_I = 1$ and $\Gamma_{\substr{H}\otimes\substr{K}} = \Gamma_\substr{H} \times \Gamma_\substr{K}$ for all $\substr{H}, \substr{K} \in \Sigma$.
\end{defn}

Given two substrates $\substr{H}, \substr{K} \in \Sigma$, we consider tasks $\Gamma_\substr{H} \rightarrow \Gamma_\substr{K}$ and ask which ones are \emph{possible} within the given theory of processes: in short, a task is possible when there is a \emph{constructor} which acting as a catalyst enables the task to be performed. Expanding on this, we come to the following definitions.

\begin{defn}
        Let $\left( \smc{C}, \Sigma, \Gamma \right)$ be a choice of substrates and consider two substrates $\substr{H}, \substr{K} \in \Sigma$. A process $f: \substr{H} \rightarrow \substr{K}$ is \textdef{task-inducing} if it maps states in $\Gamma_\substr{H}$ to states in $\Gamma_\substr{K}$:
        \[
            \forall \rho \in \Gamma_\substr{H}.\;
            f(\rho) \in \Gamma_\substr{K}
        \]
        We write $\indtask{f}$ for the task \textdef{induced by $f$}:
        \[
            \indtask{f} := \suchthat{\rho \mapsto f(\rho)}{\rho \in \Gamma_\substr{H}}
        \]
\end{defn}

\begin{defn}
    Let $\left( \smc{C}, \Sigma, \Gamma \right)$ be a choice of substrates and consider a task $\task{A}: \Gamma_\substr{H} \rightarrow \Gamma_\substr{K}$. We say that $\task{A}$ is \textdef{possible} if there are:
    \begin{itemize}
        \item[(i)] a substrate \substr{C} (acting as a \textdef{constructor} for the task)
        \item[(ii)] an attribute $P \subseteq \Gamma_\substr{C}$ (singling out the relevant constructor states)
        \item[(iii)] a task-inducing process $f: \substr{H} \otimes \substr{C} \rightarrow \substr{K} \otimes \substr{C}$ (actually performing the task)
    \end{itemize}
    such that the following two conditions are satisfied:
    \begin{enumerate}
        \item Task $\task{A}$ is obtained from the induced task $\indtask{f}$ by requiring that the input constructor state has attribute $P$ and discarding the constructor output:
        \begin{align*}
            \task{A} &=
            \substack{
               \input{def28.tikz}\\
                (\id{\Gamma_\substr{K}} \times \epsilon_{\Gamma_{\substr{C}}})
                \circ \indtask{f}
                \circ (\id{\Gamma_\substr{H}} \times P)
            }\\
            &= \suchthat{
                \rho \mapsto \rho'
            }{
                \exists \gamma \in P, \gamma' \in \Gamma_\substr{C}.\,
                f(\rho \otimes \gamma) = \rho' \otimes \gamma'
            }
        \end{align*}
        \item The attribute $P$ is preserved by the induced task $\indtask{f}$. While a particular constructor state $\gamma \in P$ may be modified to become $\gamma'$ by the underlying process of the induced task $\indtask{f}$, $\gamma'$ remains a constructor state for the same induced task $\indtask{f}$, i.e. $\gamma' \in P$. In $\Rel$, this constraint is equivalently expressed as the induced task $\indtask{f}$ sending the set of constructors $P$ to a subset of itself, regardless of the input and output on the substrates $\substr{H},\substr{K}$:
        \[
            \substack{
                \input{def282.tikz}\\
                (\epsilon_{\Gamma_\substr{K}} \times \id{\Gamma_{\substr{C}}})
                \circ \indtask{f}
                \circ (\eta_{\Gamma_\substr{H}} \times P) \subseteq P
            }
        \]
    \end{enumerate}
    We write $\possibletasks{\left( \smc{C}, \Sigma, \Gamma \right)}$ for the set of possible tasks under the given choice of substrates.
\end{defn}

The main result of this section is that possible tasks for a choice of substrate form a sub-SMC of $\Rel$, i.e. that they are closed under composition in arbitrary (acyclic) networks.

\begin{proposition}
The possible tasks $\possibletasks{\left( \smc{C}, \Sigma, \Gamma \right)}$ for a given choice of substrates form a sub-SMC of $\Rel$.
\end{proposition}
\begin{proof}
Write $\smc{C} = \left(\obj{\smc{C}}, \otimes, I\right)$.
The identity tasks and swap tasks for all systems are made possible by the identity and symmetry isomorphisms of $\smc{C}$, with trivial constructor $\substr{C} := I$:
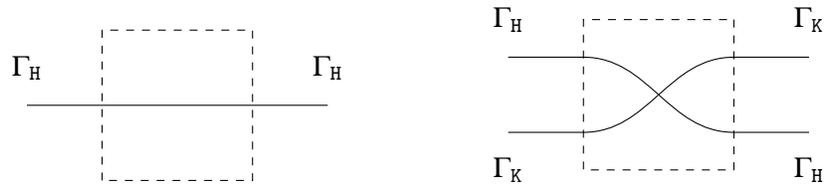
\[
    \input{idtask.tikz}
    \hspace{2cm}
    \input{swaptask.tikz}
\]
The sequential composition $\task{B} \circ \task{A}$ of possible tasks $\task{A}$ and $\task{B}$, with constructors $C$ and $D$ respectively, is possible with constructor $C \otimes D$:
\footnote{This step of the proof becomes more complicated if constructors are forced to have individual identities (i.e. in a partially monoidal category) and the same constructor must be reused by task $\task{B}$ after being used by task $\task{A}$. We leave the handling of this more sophisticated process-theoretic interpretation of constructor theory to future work.}
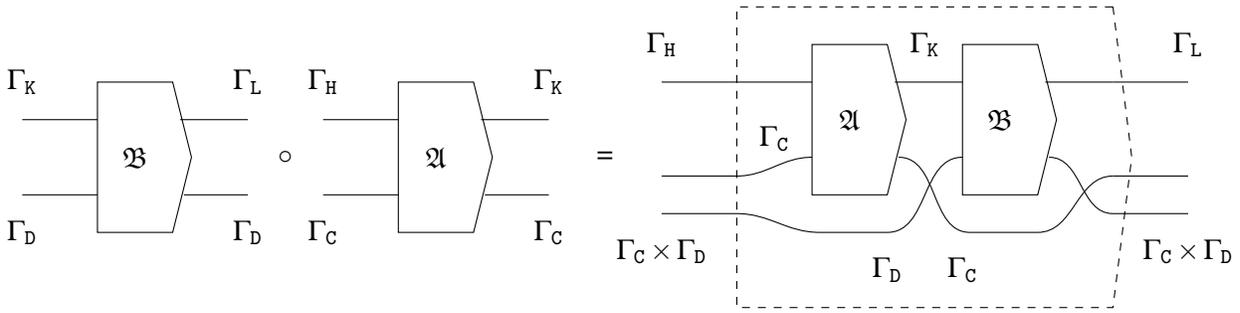
\[\input{seqtaskcomp.tikz}\]
The parallel composition $\task{B} \times \task{A}$ of possible tasks $\task{A}$ and $\task{B}$, with constructors $C$ and $D$ respectively, is possible with constructor $C \otimes D$:
\footnote{This step of the proof becomes more complicated if constructors are forced to have individual identities and the same constructor must be simultaneously used by task $\task{B}$ and task $\task{A}$. We leave the handling of this more sophisticated process-theoretic interpretation of constructor theory to future work.}
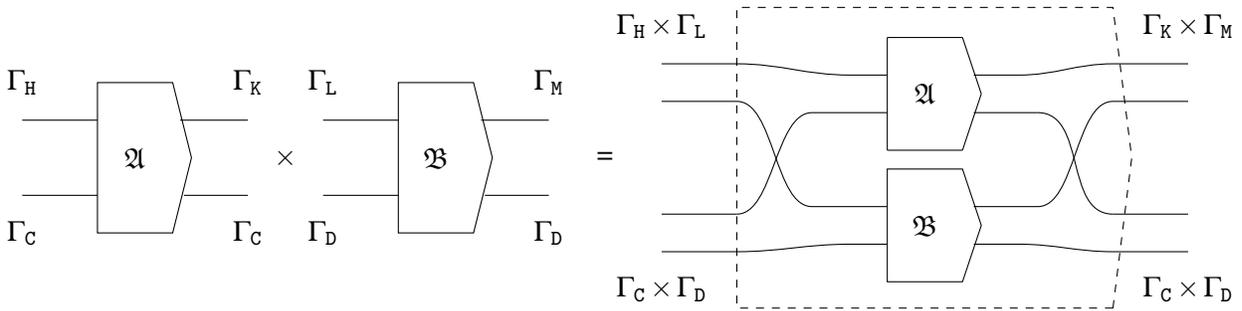
\[\input{partaskcomp.tikz}\]
This completes our proof.
\end{proof}

\section{Attributes as states}

More modern perspectives in constructor theory argue that tasks should be defined on the attributes of a substrate, rather than on the underlying states.
This captures the idea that the abstract specification of (possible) tasks---the basis upon which constructor theorists judge other theories of physics---should be based on the observable ``macrostates'' of a physical system (attributes/subsets of a set), rather than on the unobserved ``microstates'' which constitute them (states/elements of a set).
In this section, we show how the attribute-based perspective can be derived from the state-based perspective, in a compositionally sound way, by performing a suitable coarse-graining.

To start with, we define a notion of ``coarse-graining'' for tasks, moving from tasks defined on states (the ``microstates'', to stick to the thermodynamical metaphor) to tasks defined on attributes (the ``macrostates'', using the same metaphor).
We allow for the attributes involved to have non-trivial overlap---that is, we don't ask for them to form a partition---but we disallow nesting $S \subset T$ of different attributes; formally, we require for the set of attributes involved to form an ``antichain'' in the inclusion order $\subseteq$.

\begin{defn}
    Let $X$ be a set. A set $\bar{X} \subseteq \mathcal{P}(X)$ of attributes on $X$ is an \textdef{antichain} if no two attributes are nested into each other:
    \[
        \forall S, T \in \bar{X}.\, S \subseteq T \Rightarrow S = T
    \]
\end{defn}

Having fixed a choice of attributes $\bar{X}$ on $X$ and $\bar{Y}$ on $Y$, any task $\task{A}: X \rightarrow Y$ induces a ``coarse-grained task'' on the sets of attributes, as follows: for attributes $S \in \bar{X}$ and $T \in \bar{Y}$, we say that $S \mapsto T$ in the coarse-grained task if whenever an input state $x \in X$ has attribute $S$, i.e. whenever $x \in S$, then at least one of the possible outputs states $\suchthat{y \in Y}{x \tmapsto{A} y}$ has attribute $T$, i.e. $\exists y \in T.\, x \tmapsto{A} y$.
Put it another way, $S \mapsto T$ in the coarse-grained task means that the output of task $\task{A}$ \emph{can have} attribute $T$ whenever the input \emph{has} attribute $S$.

\begin{defn}
    Let $\task{A}: X \rightarrow Y$ be a task.
    Let $\bar{X} \subseteq \mathcal{P}(X)$ and $\bar{Y} \subseteq \mathcal{P}(Y)$ be sets of attributes of $X$ and $Y$ respectively.
    Then the \textdef{coarse-grained task} $\quotask{\task{A}}{\bar{X}}{\bar{Y}}: \bar{X} \rightarrow \bar{Y}$ is defined as follows:
    \[
        \quotask{\task{A}}{\bar{X}}{\bar{Y}} :=
        \left\{
            S \mapsto T
        \;\middle|\;
            S \in \bar{X},\, T \in \bar{Y},\, S \subseteq \task{A}^\dagger \circ T
        \right\}
    \]
\end{defn}

We conclude this section with three results, piecing the coarse-graining story together.
Firstly, we prove that given any process theory of tasks---including, amongst many others, the theory of all conceivable tasks and all theories of possible tasks---the coarse-grainings of the tasks can themselves be arranged into a process theory.
This shows that tasks defined on attributes are just as compositionally sound as those defined on states.
Secondly, we remark how the original ordinary tasks, defined on states, can be compositionally embedded into the universe of coarse-grained tasks, proving that the latter are a sound generalisation of the former.
Finally, we remark that coarse-grained tasks can be embedded back into the universe of ordinary tasks, proving that ordinary tasks are as expressive as coarse-grained ones.

\begin{proposition}
    Let $\mathcal{C}$ be a sub-SMC of $\Rel$, i.e. a collection of systems and tasks closed under parallel and sequential composition.
    The following defines a SMC $\overline{\mathcal{C}}$, which we refer to as the theory of \textdef{coarse-grained tasks} associated to $\mathcal{C}$:
    \begin{itemize}
        \item objects are all possible antichains of attributes for all possible sets of states:
        \[
            \obj{\overline{\mathcal{C}}} :=
            \bigcup_{X \in \obj{\mathcal{C}}} \left\{
                \bar{X} \subseteq \mathcal{P}(X)
            \;\middle|\;
                \bar{X} \text{ antichain}
            \right\}
        \]
        \item morphisms $\bar{X} \rightarrow \bar{Y}$ in $\bar{\mathcal{C}}$ are coarse-grained tasks corresponding to tasks $X \rightarrow Y$:
        \[
            \overline{\mathcal{C}}\left(\bar{X}, \bar{Y}\right)
            := \left\{
                \quotask{\task{A}}{\bar{X}}{\bar{Y}}
            \;\middle|\;
                \task{A}: X \stackrel{\mathcal{C}}{\longrightarrow} Y
            \right\}
        \]
        \item sequential composition $\circ$ is inherited from $\mathcal{C}$
        \item parallel composition $\boxtimes$ on objects is defined as:
        \[
            \bar{X} \boxtimes \bar{Y} := \suchthat{S \times T}{S \in \bar{X},\, T \in \bar{Y}}
        \]
        \item parallel composition $\boxtimes$ on morphisms arises by coarse-graining from that of $\mathcal{C}$:
        \[
            \quotask{\task{A}}{\bar{X}}{\bar{Y}}
            \boxtimes \quotask{\task{B}}{\bar{Z}}{\bar{W}}
            := \quotask{\left(\task{A} \times \task{B}\right)}{\bar{X}\boxtimes\bar{Z}}{\bar{Y}\boxtimes\bar{W}}
        \]
        \item identity and symmetry isomorphisms arise by coarse-graining from those of $\mathcal{C}$:
        \[
            \quotask{\id{X}}{\bar{X}}{\bar{X}} = \id{\bar{X}}
            \hspace{2cm}
            \quotask{\sigma_{X,Y}}{\bar{X}\boxtimes\bar{Y}}{\bar{Y}\boxtimes\bar{X}} = \sigma_{\bar{X}, \bar{Y}}
        \]
    \end{itemize}
    In particular, morphisms are well-defined, i.e. whenever $\bar{X} = \bar{X'}$ and $\bar{Y} = \bar{Y'}$ we have:
    \[
        \suchthat{\bar{\task{A}}}{\task{A}: X \rightarrow Y}
        = \suchthat{\bar{\task{A}}}{\task{A}: X' \rightarrow Y'}
    \]
\end{proposition}
\begin{proof}
    Objects are clearly well-defined, but well-definition of morphisms requires proof.
    Let $X$ and $Y$ be sets, let $\bar{X} \subseteq \mathcal{P}(X)$ and $\bar{Y} \subseteq \mathcal{P}(Y)$ be antichains.
    It suffices to show the following for $X' := \bigcup \bar{X} \subseteq X$ and $Y' := \bigcup \bar{Y} \subseteq Y$:
    \[
            \left\{
                \quotask{\task{A}}{\bar{X}}{\bar{Y}}
            \;\middle|\;
                \task{A}: X \rightarrow Y
            \right\}
            = \left\{
                \quotask{\task{A}}{\bar{X}}{\bar{Y}}
            \;\middle|\;
                \task{A}: X' \rightarrow Y'
            \right\}
    \]
    Write $\pi_{X'}:= \{x \mapsto x | x \in X'\}: X \rightarrow X$ and $\pi_{Y'}:= \{y \mapsto y | y \in Y'\}: Y \rightarrow Y$.
    If $S \in \bar{X}$ and $T \in \bar{Y}$, then $S \subseteq X'$ and $T \subseteq Y'$, and hence:
    \[
        S \subseteq \task{A}^\dagger \circ T
        \;\Leftrightarrow\; S \subseteq \pi_{X'} \circ \task{A}^\dagger \circ T
        \;\Leftrightarrow\; S \subseteq \pi_{X'} \circ \task{A}^\dagger \circ \pi_{Y'} \circ T
    \]
    Observing that $\pi_{Y'} \circ \task{A} \circ \pi_{X'}$ is a task $X' \rightarrow Y'$ completes the proof that morphisms are well-defined.
    For identities, we want to show that $\quotask{\id{X}}{\bar{X}}{\bar{X}} = \id{\bar{X}}$, and this is exactly the definition of $\bar{X}$ being an antichain: 
    \[
        \quotask{\id{X}}{\bar{X}}{\bar{X}} = \id{\bar{X}}
        \Leftrightarrow \left[
            S \subseteq T \Rightarrow S = T
        \right]
    \]
    For symmetry isomorphisms, we want to show that $\quotask{\sigma_{X,Y}}{\bar{X}\times\bar{Y}}{\bar{X}\times\bar{Y}} = \sigma_{\bar{X}, \bar{Y}}$, and this again follows from the antichain requirement:
    \begin{align*}
        \quotask{\sigma_{X,Y}}{\bar{X}\boxtimes\bar{Y}}{\bar{X}\boxtimes\bar{Y}} = \sigma_{\bar{X}, \bar{Y}}
        &\Leftrightarrow\left[
            \left(S \times T \subseteq \sigma_{X,Y}^\dagger \circ (T' \times S')\right)
            \Rightarrow \left(S = S' \text{ and } T = T'\right)
        \right]
        \\
        &\Leftrightarrow\left[
            \left(S \subseteq S' \text{ and } T \subseteq T'\right)
            \Rightarrow \left(S = S' \text{ and } T = T'\right)
        \right]
    \end{align*}
    For sequential composition to be well-defined, we need to show that $S \subseteq \task{A}^\dagger \circ T$ and $T \subseteq \task{B}^\dagger \circ U$ imply $S \subseteq (\task{B}\circ \task{A})^\dagger \circ U$:
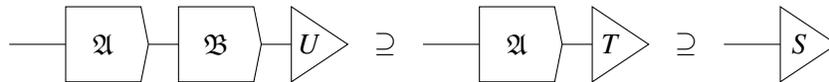
    \[\input{seqcompcheck.tikz}\]
    For parallel composition to be well-defined, we need to show that $S \subseteq \task{A}^\dagger \circ U$ and $T \subseteq \task{B}^\dagger \circ V$ imply $S \times T \subseteq (\task{A}\boxtimes \task{B})^\dagger \circ (U \times V)$:
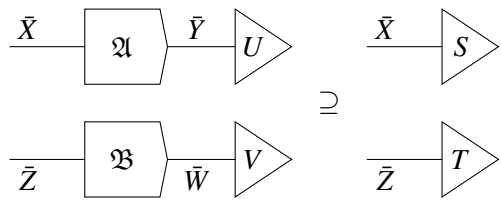
    \[\input{parcompcheck.tikz}\]
    The remaining checks are all straightforward, on similar lines.
\end{proof}

\begin{remark}
    Any sub-SMC $\mathcal{C}$ of $\Rel$ embeds into the associated theory of coarse-grained tasks $\overline{\mathcal{C}}$.
    The embedding is the functor---faithful and injective on objects---defined by sending each set to the set of its singleton subsets:
    \begin{align*}
        F(X) &:= \suchthat{\{x\}}{x \in X}\\
        F\left(\task{A}: X \rightarrow Y\right) &:= \quotask{\task{A}}{\bar{X}}{\bar{Y}} = \suchthat{
            \{x\} \mapsto \{y\}
        }{
            x \tmapsto{\task{A}} y
        }
    \end{align*}
    It is straightforward to check that the mapping defined above is a strict monoidal functor, i.e. that it preserves both sequential and parallel composition exactly (as well as identities, and symmetry isomorphisms, in this case).
\end{remark}

\begin{remark}
    Let $\mathcal{C}$ be a sub-SMC of $\Rel$.
    The associated theory of coarse-grained tasks $\overline{\mathcal{C}}$ embeds back into $\Rel$, via the identity functor:
    \[
        F(\bar{X}) := \bar{X}
        \hspace{2cm}
        F\left(\quotask{\task{A}}{\bar{X}}{\bar{Y}}\right) := \quotask{\task{A}}{\bar{X}}{\bar{Y}}
    \]
    The functor is strict monoidal when restricted to (the embedding of) $\mathcal{C}$ (into $\overline{\mathcal{C}}$).
    It is not strict (or strong) monoidal in general, because the tensor product on sets of attributes is not the same as the tensor product on sets of states:
    \[
        \bar{X} \boxtimes \bar{Y} := \suchthat{S \times T}{S \in \bar{X},\, T \in \bar{Y}}
        \neq \suchthat{(S, T)}{S \in \bar{X},\, T \in \bar{Y}} = \bar{X} \times \bar{Y}
    \]
    It is, however, lax monoidal, with the following structure morphisms:
    \[
        \left[(S, T) \mapsto S \times T\right]: \bar{X} \times \bar{Y} \rightarrow \bar{X} \boxtimes \bar{Y}
        \hspace{2cm}
        \left[\{*\} \mapsto *\right]: \bar{1} \rightarrow 1
    \]
    To see this, it suffices to observe that not only do $S \subseteq \task{A}^\dagger \circ U$ and $T \subseteq \task{B}^\dagger \circ V$ imply $S \times T \subseteq (\task{A}\boxtimes \task{B})^\dagger \circ (U \times V)$, but also $S \times T \subseteq (\task{A}\boxtimes \task{B})^\dagger \circ (U \times V)$ implies both $S \subseteq \task{A}^\dagger \circ U$ and $T \subseteq \task{B}^\dagger \circ V$.
\end{remark}





\section{Conclusion and historical remarks}

We have given categorical semantics for constructor theory in its most general form, interpreting to the best of our ability the desired mathematical foundations both set out in Deutsch's original paper \cite{deutsch2013constructor} and expressed to us by current practitioners.
We remark, without further comment, that the diagrammatic syntax we have used to formally incarnate constructor theory is also interpretable in other symmetric monoidal categories.
A long form presentation of the same content with worked examples from the constructor theory literature is in preparation.

We close with two historical case studies intended to inform constructor theorists of the topically-relevant history of process theories as applied to quantum theory, and to encourage the pursuit of the possible-impossible dichotomy by illustrating some of the fruitful outcomes that may result.

\paragraph{Process theories arose from counterfactual reasoning.}

\emph{Possibility}, read as what \emph{could} happen, is at the heart of constructor theory: here, constructor theory and process theories share a lineage of counterfactual reasoning, tracing back to Aristotle's distinction between ``actual'' and ``potential''.
One ancestor of process theories along this lineage is the Geneva school of quantum logic \cite{Piron, Jauch}, which defined the properties of physical systems in terms of experiments that could be performed \cite{DJMoore}, resulting in the linearity of physical processes \cite{FMP} due to an adjunction between cause and consequence (cf.~weakest precondition semantics in computer science \cite{hoare1987weakest}).
This led to the development of a process-theoretic framework for quantum theory, which encoded the structural consequences of an adjunction between causes and consequences in terms of a quantaloid \cite{CMS}.
The underlying structure of spaces (= quantum logics) was induced at the level of processes, and efforts were made to cast the composition of systems in those terms through process-state duality \cite{coecke2000}.
However, the current success of process theories relies on dumping quantum logics and replacing them with specially chosen processes (cf.~cups and caps \cite{AC1}).
A process theory, when formulated as a concrete symmetric monoidal category, is about possible and impossible processes that obey the axioms of the corresponding category.
Reconstructions of quantum theory in terms of process theories turn these categorical axioms into physical postulates that are considered more reasonable by some \cite{HardyJTF, selby2018reconstructing}.

%
%
%
%
%

\paragraph{Quantum from no-cloning.}

In constructor theory, the cut between possible and impossible tasks is used to define theories, and it has been suggested that the impossibility to clone should yield quantum theory, at least in a broad sense.
In categorical quantum mechanics \cite{AC1, coecke_picturing_2017, coecke_quantum_2023}, dating back to at least 2006, classicality was indeed defined by the ability to clone \cite{CPav}: this has resulted in the development of spiders \cite{CPaqPav} and the ZX-calculus \cite{CD2}, now a prominent formalism in quantum foundations, quantum computation, and general education on quantum theory.

%

\section*{Acknowledgements}

With many thanks to \textbf{Maria Violaris} and \textbf{Anicet Tibau Vidal} for their clarity and patience when presenting constructor theory at the Wolfson quantum foundations discussion, and to \textbf{Nicola Pinzani} for his insights in conversation with one of the authors.

\bibliographystyle{eptcs}
\bibliography{bibliography}
\nocite{*}

\end{document}

%% file: figures/seqcomp.tikz
\begin{tikzpicture}
	\begin{pgfonlayer}{nodelayer}
		\node [style=none] (0) at (-3.25, 0) {};
		\node [style=none] (1) at (1, 0.5) {};
		\node [style=none] (2) at (1, -0.5) {};
		\node [style=none] (3) at (2, -0.5) {};
		\node [style=none] (4) at (2, 0.5) {};
		\node [style=none] (5) at (-2, 0.5) {};
		\node [style=none] (6) at (-2, -0.5) {};
		\node [style=none] (7) at (-1, -0.5) {};
		\node [style=none] (8) at (-1, 0.5) {};
		\node [style=none] (9) at (1.55, 0) {$\task{B}$};
		\node [style=none] (10) at (-1.45, 0) {$\task{A}$};
		\node [style=none] (11) at (-2, 0) {};
		\node [style=none] (12) at (-0.75, 0) {};
		\node [style=none] (13) at (1, 0) {};
		\node [style=none] (14) at (3.25, 0) {};
		\node [style=none] (15) at (2.25, 0) {};
		\node [style=none] (16) at (-3, 0.75) {$X$};
		\node [style=none] (17) at (0, 0.75) {$Y$};
		\node [style=none] (18) at (3, 0.75) {$Z$};
	\end{pgfonlayer}
	\begin{pgfonlayer}{edgelayer}
		\draw (4.center)
			 to (15.center)
			 to (3.center)
			 to (2.center)
			 to (1.center)
			 to cycle;
		\draw (8.center)
			 to (12.center)
			 to (7.center)
			 to (6.center)
			 to (5.center)
			 to cycle;
		\draw (0.center) to (11.center);
		\draw (12.center) to (13.center);
		\draw (15.center) to (14.center);
	\end{pgfonlayer}
\end{tikzpicture}

%% file: figures/parcomp.tikz
\begin{tikzpicture}
	\begin{pgfonlayer}{nodelayer}
		\node [style=none] (0) at (-0.5, 1.25) {};
		\node [style=none] (1) at (-0.5, 0.25) {};
		\node [style=none] (2) at (0.5, 0.25) {};
		\node [style=none] (3) at (0.5, 1.25) {};
		\node [style=none] (4) at (0.05, 0.75) {$\task{A}$};
		\node [style=none] (5) at (-0.5, 0.75) {};
		\node [style=none] (6) at (0.75, 0.75) {};
		\node [style=none] (7) at (-0.5, -0.25) {};
		\node [style=none] (8) at (-0.5, -1.25) {};
		\node [style=none] (9) at (0.5, -1.25) {};
		\node [style=none] (10) at (0.5, -0.25) {};
		\node [style=none] (11) at (0.05, -0.75) {$\task{B}$};
		\node [style=none] (12) at (-0.5, -0.75) {};
		\node [style=none] (13) at (0.75, -0.75) {};
		\node [style=none] (14) at (-1.75, 0.75) {};
		\node [style=none] (15) at (1.75, 0.75) {};
		\node [style=none] (16) at (-1.75, -0.75) {};
		\node [style=none] (17) at (1.75, -0.75) {};
		\node [style=none] (18) at (-1.5, 1.5) {$X$};
		\node [style=none] (19) at (1.5, 1.5) {$Y$};
		\node [style=none] (20) at (-1.5, -1.5) {$Z$};
		\node [style=none] (21) at (1.5, -1.5) {$W$};
	\end{pgfonlayer}
	\begin{pgfonlayer}{edgelayer}
		\draw (1.center)
			 to (0.center)
			 to (3.center)
			 to (6.center)
			 to (2.center)
			 to cycle;
		\draw (10.center)
			 to (13.center)
			 to (9.center)
			 to (8.center)
			 to (7.center)
			 to cycle;
		\draw (14.center) to (5.center);
		\draw (6.center) to (15.center);
		\draw (16.center) to (12.center);
		\draw (13.center) to (17.center);
	\end{pgfonlayer}
\end{tikzpicture}

%% file: figures/transpose.tikz
\begin{tikzpicture}
	\begin{pgfonlayer}{nodelayer}
		\node [style=none] (0) at (0.5, 0.5) {};
		\node [style=none] (1) at (0.5, -0.5) {};
		\node [style=none] (2) at (-0.5, -0.5) {};
		\node [style=none] (3) at (-0.5, 0.5) {};
		\node [style=none] (4) at (-0.05, 0) {$\task{A}^\dag$};
		\node [style=none] (5) at (0.5, 0) {};
		\node [style=none] (6) at (-0.75, 0) {};
		\node [style=none] (7) at (-2, 0) {};
		\node [style=none] (8) at (1.75, 0) {};
		\node [style=none] (9) at (1.5, 0.75) {$X$};
		\node [style=none] (10) at (-1.75, 0.75) {$Y$};
	\end{pgfonlayer}
	\begin{pgfonlayer}{edgelayer}
		\draw (3.center)
			 to (6.center)
			 to (2.center)
			 to (1.center)
			 to (0.center)
			 to cycle;
		\draw (7.center) to (6.center);
		\draw (5.center) to (8.center);
	\end{pgfonlayer}
\end{tikzpicture}

%% file: figures/syms.tikz
\begin{tikzpicture}
	\begin{pgfonlayer}{nodelayer}
		\node [style=none] (0) at (-0.5, 0.5) {};
		\node [style=none] (1) at (-0.5, -0.5) {};
		\node [style=none] (2) at (0.5, 0.5) {};
		\node [style=none] (3) at (0.5, -0.5) {};
		\node [style=none] (4) at (-2, 0.5) {};
		\node [style=none] (5) at (2, 0.5) {};
		\node [style=none] (6) at (-2, -0.5) {};
		\node [style=none] (7) at (2, -0.5) {};
		\node [style=none] (8) at (-0.75, 0.75) {};
		\node [style=none] (9) at (0.75, 0.75) {};
		\node [style=none] (10) at (0.75, -0.75) {};
		\node [style=none] (11) at (-0.75, -0.75) {};
		\node [style=none] (12) at (-1.75, 1.25) {$X$};
		\node [style=none] (13) at (-1.75, -1.25) {$Y$};
		\node [style=none] (14) at (1.75, 1.25) {$Y$};
		\node [style=none] (15) at (1.75, -1.25) {$X$};
	\end{pgfonlayer}
	\begin{pgfonlayer}{edgelayer}
		\draw (6.center)
			 to (1.center)
			 to [in=-180, out=0] (2.center)
			 to (5.center);
		\draw (4.center)
			 to (0.center)
			 to [in=180, out=0] (3.center)
			 to (7.center);
		\draw [style=H] (8.center)
			 to (9.center)
			 to (10.center)
			 to (11.center)
			 to cycle;
	\end{pgfonlayer}
\end{tikzpicture}

%% file: figures/symeqs.tikz
\begin{tikzpicture}
	\begin{pgfonlayer}{nodelayer}
		\node [style=none] (0) at (-7.75, 2) {$\forall XY$};
		\node [style=none] (1) at (-6.25, 2.5) {};
		\node [style=none] (2) at (-6.25, 1.5) {};
		\node [style=none] (3) at (-5.25, 2.5) {};
		\node [style=none] (4) at (-5.25, 1.5) {};
		\node [style=none] (5) at (-7, 2.5) {};
		\node [style=none] (6) at (-4.75, 2.5) {};
		\node [style=none] (7) at (-7, 1.5) {};
		\node [style=none] (8) at (-4.75, 1.5) {};
		\node [style=none] (9) at (-4.25, 2.5) {};
		\node [style=none] (10) at (-4.25, 1.5) {};
		\node [style=none] (11) at (-3.25, 2.5) {};
		\node [style=none] (12) at (-3.25, 1.5) {};
		\node [style=none] (13) at (-4.75, 2.5) {};
		\node [style=none] (14) at (-2.5, 2.5) {};
		\node [style=none] (15) at (-4.75, 1.5) {};
		\node [style=none] (16) at (-2.5, 1.5) {};
		\node [style=none] (17) at (-1.75, 2) {$=$};
		\node [style=none] (18) at (-1, 2.5) {};
		\node [style=none] (19) at (-1, 1.5) {};
		\node [style=none] (20) at (1, 2.5) {};
		\node [style=none] (21) at (1, 1.5) {};
		\node [style=none] (22) at (-6.75, 3.25) {$X$};
		\node [style=none] (23) at (-6.75, 0.75) {$Y$};
		\node [style=none] (24) at (-2.75, 3.25) {$X$};
		\node [style=none] (25) at (-2.75, 0.75) {$Y$};
		\node [style=none] (26) at (-4.75, 0.75) {$X$};
		\node [style=none] (27) at (-4.75, 3.25) {$Y$};
		\node [style=none] (28) at (0, 3.25) {$X$};
		\node [style=none] (29) at (0, 0.75) {$Y$};
		\node [style=none] (30) at (-4, -1) {$\forall WXYZ$};
		\node [style=none] (31) at (-0.25, -1.5) {};
		\node [style=none] (32) at (0.75, -3) {};
		\node [style=none] (33) at (0.75, -1.5) {};
		\node [style=none] (34) at (-0.25, -3) {};
		\node [style=none] (44) at (-4, -2) {$\forall \task{A} : X \rightarrow Y$};
		\node [style=none] (45) at (-4, -3) {$\forall \task{B} : Z \rightarrow W$};
		\node [style=none] (46) at (-1.5, -2.5) {};
		\node [style=none] (47) at (-1.5, -3.5) {};
		\node [style=none] (48) at (-0.5, -3.5) {};
		\node [style=none] (49) at (-0.5, -2.5) {};
		\node [style=none] (50) at (-1.5, -1) {};
		\node [style=none] (51) at (-1.5, -2) {};
		\node [style=none] (52) at (-0.5, -2) {};
		\node [style=none] (53) at (-0.5, -1) {};
		\node [style=none] (54) at (-0.95, -3) {$\task{B}$};
		\node [style=none] (55) at (-0.95, -1.5) {$\task{A}$};
		\node [style=none] (56) at (-1.5, -1.5) {};
		\node [style=none] (57) at (-0.25, -1.5) {};
		\node [style=none] (58) at (-1.5, -3) {};
		\node [style=none] (59) at (-0.25, -3) {};
		\node [style=none] (60) at (3.75, -1.5) {};
		\node [style=none] (61) at (4.75, -3) {};
		\node [style=none] (62) at (4.75, -1.5) {};
		\node [style=none] (63) at (3.75, -3) {};
		\node [style=none] (64) at (4.75, -1) {};
		\node [style=none] (65) at (4.75, -2) {};
		\node [style=none] (66) at (5.75, -2) {};
		\node [style=none] (67) at (5.75, -1) {};
		\node [style=none] (68) at (4.75, -2.5) {};
		\node [style=none] (69) at (4.75, -3.5) {};
		\node [style=none] (70) at (5.75, -3.5) {};
		\node [style=none] (71) at (5.75, -2.5) {};
		\node [style=none] (72) at (5.3, -1.5) {$\task{B}$};
		\node [style=none] (73) at (5.3, -3) {$\task{A}$};
		\node [style=none] (74) at (4.75, -3) {};
		\node [style=none] (75) at (6, -3) {};
		\node [style=none] (76) at (4.75, -1.5) {};
		\node [style=none] (77) at (6, -1.5) {};
		\node [style=none] (78) at (-2.5, -1.5) {};
		\node [style=none] (79) at (-2.5, -3) {};
		\node [style=none] (80) at (1.5, -1.5) {};
		\node [style=none] (81) at (1.5, -3) {};
		\node [style=none] (82) at (3, -1.5) {};
		\node [style=none] (83) at (3, -3) {};
		\node [style=none] (84) at (7, -1.5) {};
		\node [style=none] (85) at (7, -3) {};
		\node [style=none] (86) at (2.25, -2.25) {$=$};
		\node [style=none] (87) at (-2.25, -0.75) {$X$};
		\node [style=none] (88) at (1.25, -3.75) {$Y$};
		\node [style=none] (89) at (-2.25, -3.75) {$Z$};
		\node [style=none] (90) at (3.25, -3.75) {$Z$};
		\node [style=none] (91) at (3.25, -0.75) {$X$};
		\node [style=none] (92) at (6.75, -3.75) {$Y$};
		\node [style=none] (93) at (6.75, -0.75) {$W$};
		\node [style=none] (94) at (1.25, -0.75) {$W$};
	\end{pgfonlayer}
	\begin{pgfonlayer}{edgelayer}
		\draw (7.center) to (2.center);
		\draw [in=-180, out=0] (2.center) to (3.center);
		\draw (3.center) to (6.center);
		\draw (5.center) to (1.center);
		\draw [in=180, out=0] (1.center) to (4.center);
		\draw (4.center) to (8.center);
		\draw (15.center) to (10.center);
		\draw [in=-180, out=0] (10.center) to (11.center);
		\draw (11.center) to (14.center);
		\draw (13.center) to (9.center);
		\draw [in=180, out=0] (9.center) to (12.center);
		\draw (12.center) to (16.center);
		\draw (18.center) to (20.center);
		\draw (19.center) to (21.center);
		\draw (34.center)
			 to [in=-180, out=0] (33.center)
			 to (80.center);
		\draw (31.center)
			 to [in=180, out=0] (32.center)
			 to (81.center);
		\draw (49.center)
			 to (59.center)
			 to (48.center)
			 to (47.center)
			 to (46.center)
			 to cycle;
		\draw (53.center)
			 to (57.center)
			 to (52.center)
			 to (51.center)
			 to (50.center)
			 to cycle;
		\draw (83.center)
			 to (63.center)
			 to [in=-180, out=0] (62.center);
		\draw (82.center)
			 to (60.center)
			 to [in=180, out=0] (61.center);
		\draw (67.center)
			 to (77.center)
			 to (66.center)
			 to (65.center)
			 to (64.center)
			 to cycle;
		\draw (71.center)
			 to (75.center)
			 to (70.center)
			 to (69.center)
			 to (68.center)
			 to cycle;
		\draw (79.center) to (58.center);
		\draw (78.center) to (56.center);
		\draw (75.center) to (85.center);
		\draw (77.center) to (84.center);
	\end{pgfonlayer}
\end{tikzpicture}

%% file: figures/copydel.tikz
\begin{tikzpicture}
	\begin{pgfonlayer}{nodelayer}
		\node [style=none] (0) at (-4.75, 0) {};
		\node [style=smalldotb] (2) at (-3.5, 0) {};
		\node [style=none] (3) at (-2.75, 0.5) {};
		\node [style=none] (4) at (-2.75, -0.5) {};
		\node [style=none] (5) at (-5.5, 0) {$X$};
		\node [style=none] (6) at (-1.5, 0.5) {$X$};
		\node [style=none] (7) at (-1.5, -0.5) {$X$};
		\node [style=none] (8) at (-2.25, 0.5) {};
		\node [style=none] (9) at (-2.25, -0.5) {};
		\node [style=none] (10) at (-3.5, -1.5) {$\delta_X$};
		\node [style=none] (11) at (-2.75, 0.75) {};
		\node [style=none] (12) at (-2.75, -0.75) {};
		\node [style=none] (13) at (-4.25, -0.75) {};
		\node [style=none] (14) at (-4.25, 0.75) {};
		\node [style=none] (15) at (2.75, -1.5) {$\epsilon_X$};
		\node [style=none] (16) at (2.25, 0) {};
		\node [style=smalldotb] (17) at (3.5, 0) {};
		\node [style=none] (22) at (4.25, 0.75) {};
		\node [style=none] (23) at (4.25, -0.75) {};
		\node [style=none] (24) at (2.75, -0.75) {};
		\node [style=none] (25) at (2.75, 0.75) {};
		\node [style=none] (26) at (1.5, 0) {$X$};
	\end{pgfonlayer}
	\begin{pgfonlayer}{edgelayer}
		\draw (0.center) to (2);
		\draw (2)
			 to [in=180, out=90] (3.center)
			 to (8.center);
		\draw (2)
			 to [in=-180, out=-90] (4.center)
			 to (9.center);
		\draw [style=H] (14.center) to (11.center);
		\draw [style=H] (11.center) to (12.center);
		\draw [style=H] (12.center) to (13.center);
		\draw [style=H] (13.center) to (14.center);
		\draw (16.center) to (17);
		\draw [style=H] (25.center) to (22.center);
		\draw [style=H] (22.center) to (23.center);
		\draw [style=H] (23.center) to (24.center);
		\draw [style=H] (24.center) to (25.center);
	\end{pgfonlayer}
\end{tikzpicture}

%% file: figures/copydeleqs.tikz
\begin{tikzpicture}
	\begin{pgfonlayer}{nodelayer}
		\node [style=none] (4) at (-7.25, 2) {};
		\node [style=smalldotb] (5) at (-6.25, 2) {};
		\node [style=none] (6) at (-5.75, 2.5) {};
		\node [style=none] (7) at (-5.75, 1.5) {};
		\node [style=none] (8) at (-4.75, 2.5) {};
		\node [style=none] (9) at (-4.75, 1.5) {};
		\node [style=none] (10) at (-3.25, 2) {};
		\node [style=smalldotb] (11) at (-2.25, 2) {};
		\node [style=none] (12) at (-1.75, 2.5) {};
		\node [style=none] (13) at (-1.75, 1.5) {};
		\node [style=none] (14) at (-1.25, 2.5) {};
		\node [style=none] (15) at (-1.25, 1.5) {};
		\node [style=none] (16) at (-4, 2) {$=$};
		\node [style=smalldotb] (18) at (-1.5, 0.5) {};
		\node [style=none] (19) at (-0.75, 1) {};
		\node [style=none] (20) at (-0.75, 0) {};
		\node [style=none] (21) at (-3.25, -0.2) {};
		\node [style=smalldotb] (22) at (-2.25, -0.2) {};
		\node [style=none] (23) at (-1.5, 0.5) {};
		\node [style=none] (24) at (-1.5, -1) {};
		\node [style=none] (25) at (-0.75, -1) {};
		\node [style=none] (34) at (0, 0) {$=$};
		\node [style=smalldotb] (35) at (2.5, -0.5) {};
		\node [style=none] (36) at (3.25, -1) {};
		\node [style=none] (37) at (3.25, 0) {};
		\node [style=none] (38) at (0.75, 0.2) {};
		\node [style=smalldotb] (39) at (1.75, 0.2) {};
		\node [style=none] (40) at (2.5, -0.5) {};
		\node [style=none] (41) at (2.5, 1) {};
		\node [style=none] (42) at (3.25, 1) {};
		\node [style=smalldotb] (44) at (3, -2) {};
		\node [style=none] (45) at (3.75, -1.5) {};
		\node [style=none] (46) at (3.75, -2.5) {};
		\node [style=none] (47) at (2, -2) {};
		\node [style=smalldotb] (48) at (3.75, -2.5) {};
		\node [style=none] (49) at (4.25, -1.5) {};
		\node [style=none] (50) at (5, -2) {$=$};
		\node [style=none] (51) at (5.75, -2) {};
		\node [style=none] (52) at (7.75, -2) {};
	\end{pgfonlayer}
	\begin{pgfonlayer}{edgelayer}
		\draw (4.center) to (5);
		\draw [in=180, out=90] (5) to (6.center);
		\draw [in=-180, out=-90] (5) to (7.center);
		\draw [in=180, out=0] (6.center) to (9.center);
		\draw [in=-180, out=0] (7.center) to (8.center);
		\draw (10.center) to (11);
		\draw (11)
			 to [in=180, out=90] (12.center)
			 to (14.center);
		\draw (11)
			 to [in=-180, out=-90] (13.center)
			 to (15.center);
		\draw [in=180, out=90] (18) to (19.center);
		\draw [in=-180, out=-90] (18) to (20.center);
		\draw (21.center) to (22);
		\draw [in=180, out=90] (22) to (23.center);
		\draw (22)
			 to [in=-180, out=-90] (24.center)
			 to (25.center);
		\draw [in=-180, out=-90] (35) to (36.center);
		\draw [in=180, out=90] (35) to (37.center);
		\draw (38.center) to (39);
		\draw [in=-180, out=-90] (39) to (40.center);
		\draw (39)
			 to [in=180, out=90] (41.center)
			 to (42.center);
		\draw (44)
			 to [in=180, out=90] (45.center)
			 to (49.center);
		\draw [in=-180, out=-90] (44) to (46.center);
		\draw (47.center) to (44);
		\draw (51.center) to (52.center);
	\end{pgfonlayer}
\end{tikzpicture}

%% file: figures/match.tikz
\begin{tikzpicture}
	\begin{pgfonlayer}{nodelayer}
		\node [style=none] (0) at (1.25, 0) {};
		\node [style=smalldotb] (1) at (0, 0) {};
		\node [style=none] (2) at (-0.75, 0.5) {};
		\node [style=none] (3) at (-0.75, -0.5) {};
		\node [style=none] (4) at (2, 0) {$X$};
		\node [style=none] (5) at (-2, 0.5) {$X$};
		\node [style=none] (6) at (-2, -0.5) {$X$};
		\node [style=none] (7) at (-1.25, 0.5) {};
		\node [style=none] (8) at (-1.25, -0.5) {};
		\node [style=none] (9) at (0, -1.5) {$\delta^\dag_X$};
		\node [style=none] (10) at (-0.75, 0.75) {};
		\node [style=none] (11) at (-0.75, -0.75) {};
		\node [style=none] (12) at (0.75, -0.75) {};
		\node [style=none] (13) at (0.75, 0.75) {};
	\end{pgfonlayer}
	\begin{pgfonlayer}{edgelayer}
		\draw (0.center) to (1);
		\draw (1)
			 to [in=0, out=90] (2.center)
			 to (7.center);
		\draw (1)
			 to [in=0, out=-90] (3.center)
			 to (8.center);
		\draw [style=H] (13.center) to (10.center);
		\draw [style=H] (10.center) to (11.center);
		\draw [style=H] (11.center) to (12.center);
		\draw [style=H] (12.center) to (13.center);
	\end{pgfonlayer}
\end{tikzpicture}

%% file: figures/attr.tikz
\begin{tikzpicture}
	\begin{pgfonlayer}{nodelayer}
		\node [style=none] (0) at (-1.5, 0) {};
		\node [style=none] (1) at (-0.5, 0.5) {};
		\node [style=none] (2) at (-0.5, -0.5) {};
		\node [style=none] (3) at (-0.85, 0) {$S$};
		\node [style=none] (4) at (0.5, 0) {};
		\node [style=none] (5) at (-0.5, 0) {};
		\node [style=none] (6) at (1.25, 0) {$X$};
	\end{pgfonlayer}
	\begin{pgfonlayer}{edgelayer}
		\draw (0.center)
			 to (1.center)
			 to (2.center)
			 to cycle;
		\draw (5.center) to (4.center);
	\end{pgfonlayer}
\end{tikzpicture}

%% file: figures/unit.tikz
\begin{tikzpicture}
	\begin{pgfonlayer}{nodelayer}
		\node [style=none] (0) at (0, -1.5) {$\eta_X$};
		\node [style=none] (1) at (0.5, 0) {};
		\node [style=smalldotb] (2) at (-0.75, 0) {};
		\node [style=none] (3) at (-1.5, 0.75) {};
		\node [style=none] (4) at (-1.5, -0.75) {};
		\node [style=none] (5) at (0, -0.75) {};
		\node [style=none] (6) at (0, 0.75) {};
		\node [style=none] (7) at (1.25, 0) {$X$};
	\end{pgfonlayer}
	\begin{pgfonlayer}{edgelayer}
		\draw (1.center) to (2);
		\draw [style=H] (6.center) to (3.center);
		\draw [style=H] (3.center) to (4.center);
		\draw [style=H] (4.center) to (5.center);
		\draw [style=H] (5.center) to (6.center);
	\end{pgfonlayer}
\end{tikzpicture}

%% file: figures/preconditioned.tikz
\begin{tikzpicture}
	\begin{pgfonlayer}{nodelayer}
		\node [style=none] (0) at (0, 1) {};
		\node [style=none] (1) at (0, -1) {};
		\node [style=none] (2) at (1, 1) {};
		\node [style=none] (3) at (1.25, 0) {};
		\node [style=none] (4) at (1, -1) {};
		\node [style=none] (5) at (0, 0.5) {};
		\node [style=none] (6) at (0, -0.5) {};
		\node [style=none] (7) at (-2, 0.5) {};
		\node [style=none] (8) at (-1, -0.5) {};
		\node [style=none] (9) at (2.25, 0) {};
		\node [style=none] (10) at (0.5, 0) {$\task{A}$};
		\node [style=none] (11) at (-1.75, 0.75) {$X$};
		\node [style=none] (12) at (-0.5, -0.75) {$Z$};
		\node [style=none] (15) at (-2, -0.5) {};
		\node [style=none] (16) at (-1, 0) {};
		\node [style=none] (17) at (-1, -1) {};
		\node [style=none] (18) at (-1.35, -0.5) {$S$};
		\node [style=none] (20) at (-1, -0.5) {};
		\node [style=none] (21) at (2, 0.25) {$Y$};
	\end{pgfonlayer}
	\begin{pgfonlayer}{edgelayer}
		\draw (4.center)
			 to (1.center)
			 to (0.center)
			 to (2.center)
			 to (3.center)
			 to cycle;
		\draw (7.center) to (5.center);
		\draw (8.center) to (6.center);
		\draw (3.center) to (9.center);
		\draw (15.center)
			 to (16.center)
			 to (17.center)
			 to cycle;
	\end{pgfonlayer}
\end{tikzpicture}

%% file: figures/prediscard.tikz
\begin{tikzpicture}
	\begin{pgfonlayer}{nodelayer}
		\node [style=none] (0) at (0, 1) {};
		\node [style=none] (1) at (0, -1) {};
		\node [style=none] (2) at (1, 1) {};
		\node [style=none] (3) at (1.25, 0) {};
		\node [style=none] (4) at (1, -1) {};
		\node [style=none] (5) at (0, 0.5) {};
		\node [style=none] (6) at (0, -0.5) {};
		\node [style=none] (7) at (-2, 0.5) {};
		\node [style=none] (8) at (-1.5, -0.5) {};
		\node [style=none] (9) at (2.25, 0) {};
		\node [style=none] (10) at (0.5, 0) {$\task{A}$};
		\node [style=none] (11) at (-1.75, 0.75) {$X$};
		\node [style=none] (12) at (-0.75, -0.75) {$Z$};
		\node [style=none] (20) at (-1.5, -0.5) {};
		\node [style=none] (21) at (2, 0.25) {$Y$};
		\node [style=smalldotb] (22) at (-1.5, -0.5) {};
	\end{pgfonlayer}
	\begin{pgfonlayer}{edgelayer}
		\draw (4.center)
			 to (1.center)
			 to (0.center)
			 to (2.center)
			 to (3.center)
			 to cycle;
		\draw (7.center) to (5.center);
		\draw (8.center) to (6.center);
		\draw (3.center) to (9.center);
	\end{pgfonlayer}
\end{tikzpicture}

%% file: figures/postconditioned.tikz
\begin{tikzpicture}
	\begin{pgfonlayer}{nodelayer}
		\node [style=none] (0) at (0.25, 1) {};
		\node [style=none] (1) at (0.25, -1) {};
		\node [style=none] (2) at (-0.75, 1) {};
		\node [style=none] (3) at (-1, 0) {};
		\node [style=none] (4) at (-0.75, -1) {};
		\node [style=none] (5) at (0.25, 0.5) {};
		\node [style=none] (6) at (0.25, -0.5) {};
		\node [style=none] (7) at (2.25, 0.5) {};
		\node [style=none] (8) at (1.25, -0.5) {};
		\node [style=none] (9) at (-2, 0) {};
		\node [style=none] (10) at (-0.25, 0) {$\task{A}$};
		\node [style=none] (11) at (2, 0.75) {$X$};
		\node [style=none] (12) at (0.75, -0.75) {$Z$};
		\node [style=none] (15) at (2.25, -0.5) {};
		\node [style=none] (16) at (1.25, 0) {};
		\node [style=none] (17) at (1.25, -1) {};
		\node [style=none] (18) at (1.6, -0.5) {$S$};
		\node [style=none] (20) at (1.25, -0.5) {};
		\node [style=none] (21) at (-1.75, 0.25) {$Y$};
	\end{pgfonlayer}
	\begin{pgfonlayer}{edgelayer}
		\draw (4.center)
			 to (1.center)
			 to (0.center)
			 to (2.center)
			 to (3.center)
			 to cycle;
		\draw (7.center) to (5.center);
		\draw (8.center) to (6.center);
		\draw (3.center) to (9.center);
		\draw (15.center)
			 to (16.center)
			 to (17.center)
			 to cycle;
	\end{pgfonlayer}
\end{tikzpicture}

%% file: figures/postdiscard.tikz
\begin{tikzpicture}
	\begin{pgfonlayer}{nodelayer}
		\node [style=none] (0) at (0.25, 1) {};
		\node [style=none] (1) at (0.25, -1) {};
		\node [style=none] (2) at (-0.75, 1) {};
		\node [style=none] (3) at (-1, 0) {};
		\node [style=none] (4) at (-0.75, -1) {};
		\node [style=none] (5) at (0.25, 0.5) {};
		\node [style=none] (6) at (0.25, -0.5) {};
		\node [style=none] (7) at (2.25, 0.5) {};
		\node [style=none] (8) at (1.75, -0.5) {};
		\node [style=none] (9) at (-2, 0) {};
		\node [style=none] (10) at (-0.25, 0) {$\task{A}$};
		\node [style=none] (11) at (2, 0.75) {$X$};
		\node [style=none] (12) at (1, -0.75) {$Z$};
		\node [style=none] (20) at (1.75, -0.5) {};
		\node [style=none] (21) at (-1.75, 0.25) {$Y$};
		\node [style=smalldotb] (22) at (1.75, -0.5) {};
	\end{pgfonlayer}
	\begin{pgfonlayer}{edgelayer}
		\draw (4.center)
			 to (1.center)
			 to (0.center)
			 to (2.center)
			 to (3.center)
			 to cycle;
		\draw (7.center) to (5.center);
		\draw (8.center) to (6.center);
		\draw (3.center) to (9.center);
	\end{pgfonlayer}
\end{tikzpicture}

%% file: figures/prepost.tikz
\begin{tikzpicture}
	\begin{pgfonlayer}{nodelayer}
		\node [style=none] (0) at (-0.5, 1) {};
		\node [style=none] (1) at (-0.5, -1) {};
		\node [style=none] (2) at (0.5, 1) {};
		\node [style=none] (3) at (0.75, 0) {};
		\node [style=none] (4) at (0.5, -1) {};
		\node [style=none] (5) at (0, 0) {$\task{A}$};
		\node [style=none] (6) at (-0.5, 0.5) {};
		\node [style=none] (7) at (-0.5, -0.5) {};
		\node [style=none] (8) at (3, 0.5) {};
		\node [style=none] (9) at (-2.75, 0.5) {};
		\node [style=none] (10) at (-1.5, -0.5) {};
		\node [style=none] (11) at (-2.5, -0.5) {};
		\node [style=none] (12) at (-1.5, 0) {};
		\node [style=none] (13) at (-1.5, -1) {};
		\node [style=none] (14) at (-1.85, -0.5) {$P$};
		\node [style=none] (16) at (-1.5, -0.5) {};
		\node [style=none] (17) at (0.6, 0.5) {};
		\node [style=none] (18) at (0.65, -0.5) {};
		\node [style=none] (19) at (1.75, -0.5) {};
		\node [style=none] (20) at (1.75, -0.5) {};
		\node [style=none] (21) at (2.75, -0.5) {};
		\node [style=none] (22) at (1.75, 0) {};
		\node [style=none] (23) at (1.75, -1) {};
		\node [style=none] (24) at (2.1, -0.5) {$Q$};
		\node [style=none] (25) at (1.75, -0.5) {};
		\node [style=none] (26) at (-2.5, 0.75) {$X$};
		\node [style=none] (27) at (2.75, 0.75) {$Y$};
		\node [style=none] (28) at (-1, -0.75) {$Z$};
		\node [style=none] (29) at (1.25, -0.75) {$W$};
	\end{pgfonlayer}
	\begin{pgfonlayer}{edgelayer}
		\draw (3.center)
			 to (4.center)
			 to (1.center)
			 to (0.center)
			 to (2.center)
			 to cycle;
		\draw (10.center) to (7.center);
		\draw (9.center) to (6.center);
		\draw (11.center)
			 to (12.center)
			 to (13.center)
			 to cycle;
		\draw (18.center) to (19.center);
		\draw (17.center) to (8.center);
		\draw (21.center) to (22.center);
		\draw (22.center) to (23.center);
		\draw (23.center) to (21.center);
	\end{pgfonlayer}
\end{tikzpicture}

%% file: figures/def28.tikz
\begin{tikzpicture}
	\begin{pgfonlayer}{nodelayer}
		\node [style=none] (0) at (0, -0.5) {};
		\node [style=none] (1) at (-1, -0.5) {};
		\node [style=none] (2) at (0, 0) {};
		\node [style=none] (3) at (0, -1) {};
		\node [style=none] (4) at (-0.35, -0.5) {$P$};
		\node [style=none] (5) at (0, -0.5) {};
		\node [style=none] (6) at (-1, 0.5) {};
		\node [style=none] (8) at (1, -0.5) {};
		\node [style=none] (9) at (1, -1) {};
		\node [style=none] (10) at (1, 1) {};
		\node [style=none] (11) at (2, 1) {};
		\node [style=none] (12) at (2.25, 0) {};
		\node [style=none] (13) at (2, -1) {};
		\node [style=none] (14) at (1, 0.5) {};
		\node [style=none] (15) at (1.5, 0) {$\indtask{f}$};
		\node [style=smalldotb] (16) at (3.25, -0.5) {};
		\node [style=none] (17) at (2.15, -0.5) {};
		\node [style=none] (18) at (4, 0.5) {};
		\node [style=none] (19) at (2.15, 0.5) {};
		\node [style=none] (20) at (-0.75, 0.75) {$\Gamma_\substr{H}$};
		\node [style=none] (21) at (3.75, 0.75) {$\Gamma_\substr{K}$};
		\node [style=none] (22) at (0.5, -0.75) {$\Gamma_\substr{C}$};
		\node [style=none] (23) at (2.75, -0.75) {$\Gamma_\substr{C}$};
	\end{pgfonlayer}
	\begin{pgfonlayer}{edgelayer}
		\draw (1.center)
			 to (2.center)
			 to (3.center)
			 to cycle;
		\draw (5.center) to (8.center);
		\draw (9.center)
			 to (13.center)
			 to (12.center)
			 to (11.center)
			 to (10.center)
			 to cycle;
		\draw (6.center) to (14.center);
		\draw (17.center) to (16);
		\draw (19.center) to (18.center);
	\end{pgfonlayer}
\end{tikzpicture}

%% file: figures/def282.tikz
\begin{tikzpicture}
	\begin{pgfonlayer}{nodelayer}
		\node [style=none] (5) at (3.75, -0.5) {};
		\node [style=none] (6) at (3.5, 0.5) {};
		\node [style=none] (8) at (2.15, -0.5) {};
		\node [style=none] (9) at (1, -1) {};
		\node [style=none] (10) at (1, 1) {};
		\node [style=none] (11) at (2, 1) {};
		\node [style=none] (12) at (2.25, 0) {};
		\node [style=none] (13) at (2, -1) {};
		\node [style=none] (14) at (2.15, 0.5) {};
		\node [style=none] (15) at (1.5, 0) {$\indtask{f}$};
		\node [style=none] (17) at (1, -0.5) {};
		\node [style=none] (19) at (1, 0.5) {};
		\node [style=none] (20) at (2.75, 0.75) {$\Gamma_\substr{K}$};
		\node [style=none] (21) at (0.5, 0.75) {$\Gamma_\substr{H}$};
		\node [style=none] (22) at (3, -0.75) {$\Gamma_\substr{C}$};
		\node [style=none] (23) at (0.5, -0.75) {$\Gamma_\substr{C}$};
		\node [style=none] (24) at (0, -0.5) {};
		\node [style=none] (25) at (-1, -0.5) {};
		\node [style=none] (26) at (0, 0) {};
		\node [style=none] (27) at (0, -1) {};
		\node [style=none] (28) at (-0.35, -0.5) {$P$};
		\node [style=none] (29) at (0, -0.5) {};
		\node [style=smalldotb] (30) at (3.5, 0.5) {};
		\node [style=smalldotb] (31) at (-0.25, 0.5) {};
		\node [style=none] (32) at (1, 0.5) {};
		\node [style=none] (33) at (4.25, 0) {$\subseteq$};
		\node [style=none] (34) at (5.75, 0) {};
		\node [style=none] (35) at (4.75, 0) {};
		\node [style=none] (36) at (5.75, 0.5) {};
		\node [style=none] (37) at (5.75, -0.5) {};
		\node [style=none] (38) at (5.4, 0) {$P$};
		\node [style=none] (39) at (5.75, 0) {};
		\node [style=none] (40) at (6.75, 0) {};
	\end{pgfonlayer}
	\begin{pgfonlayer}{edgelayer}
		\draw (5.center) to (8.center);
		\draw (9.center)
			 to (13.center)
			 to (12.center)
			 to (11.center)
			 to (10.center)
			 to cycle;
		\draw (6.center) to (14.center);
		\draw (25.center)
			 to (26.center)
			 to (27.center)
			 to cycle;
		\draw (17.center) to (29.center);
		\draw (32.center) to (31);
		\draw (35.center)
			 to (36.center)
			 to (37.center)
			 to cycle;
		\draw (39.center) to (40.center);
	\end{pgfonlayer}
\end{tikzpicture}

%% file: figures/idtask.tikz
\begin{tikzpicture}
	\begin{pgfonlayer}{nodelayer}
		\node [style=none] (0) at (-1, 1) {};
		\node [style=none] (1) at (-1, -1) {};
		\node [style=none] (2) at (1, 1) {};
		\node [style=none] (3) at (1, -1) {};
		\node [style=none] (4) at (-2, 0) {};
		\node [style=none] (5) at (2, 0) {};
		\node [style=none] (6) at (2, 0.5) {$\Gamma_\substr{H}$};
		\node [style=none] (7) at (-2, 0.5) {$\Gamma_\substr{H}$};
	\end{pgfonlayer}
	\begin{pgfonlayer}{edgelayer}
		\draw [style=H] (1.center)
			 to (0.center)
			 to (2.center)
			 to (3.center)
			 to cycle;
		\draw (4.center) to (5.center);
	\end{pgfonlayer}
\end{tikzpicture}

%% file: figures/swaptask.tikz
\begin{tikzpicture}
	\begin{pgfonlayer}{nodelayer}
		\node [style=none] (0) at (-1, 1) {};
		\node [style=none] (1) at (-1, -1) {};
		\node [style=none] (2) at (1, 1) {};
		\node [style=none] (3) at (1, -1) {};
		\node [style=none] (4) at (-1, 0.5) {};
		\node [style=none] (5) at (-1, -0.5) {};
		\node [style=none] (6) at (1, 0.5) {};
		\node [style=none] (7) at (1, -0.5) {};
		\node [style=none] (8) at (-2, 0.5) {};
		\node [style=none] (9) at (-2, -0.5) {};
		\node [style=none] (10) at (2, 0.5) {};
		\node [style=none] (11) at (2, -0.5) {};
		\node [style=none] (12) at (-2, 1) {$\Gamma_\substr{H}$};
		\node [style=none] (13) at (-2, -1) {$\Gamma_\substr{K}$};
		\node [style=none] (14) at (2, 1) {$\Gamma_\substr{K}$};
		\node [style=none] (15) at (2, -1) {$\Gamma_\substr{H}$};
	\end{pgfonlayer}
	\begin{pgfonlayer}{edgelayer}
		\draw [style=H] (2.center)
			 to (3.center)
			 to (1.center)
			 to (0.center)
			 to cycle;
		\draw (8.center)
			 to (4.center)
			 to [in=180, out=0] (7.center)
			 to (11.center);
		\draw (9.center)
			 to (5.center)
			 to [in=-180, out=0] (6.center)
			 to (10.center);
	\end{pgfonlayer}
\end{tikzpicture}

%% file: figures/seqtaskcomp.tikz
\begin{tikzpicture}
	\begin{pgfonlayer}{nodelayer}
		\node [style=none] (14) at (-3.5, 0) {};
		\node [style=none] (15) at (-3.5, -2) {};
		\node [style=none] (16) at (-2.5, 0) {};
		\node [style=none] (17) at (-2.25, -1) {};
		\node [style=none] (18) at (-2.5, -2) {};
		\node [style=none] (19) at (-3, -1) {$\task{A}$};
		\node [style=none] (20) at (-3.5, -0.5) {};
		\node [style=none] (21) at (-3.5, -1.5) {};
		\node [style=none] (22) at (-2.4, -0.5) {};
		\node [style=none] (23) at (-2.35, -1.5) {};
		\node [style=none] (24) at (-7.5, 0) {};
		\node [style=none] (25) at (-7.5, -2) {};
		\node [style=none] (26) at (-6.5, 0) {};
		\node [style=none] (27) at (-6.25, -1) {};
		\node [style=none] (28) at (-6.5, -2) {};
		\node [style=none] (29) at (-7, -1) {$\task{B}$};
		\node [style=none] (30) at (-7.5, -0.5) {};
		\node [style=none] (31) at (-7.5, -1.5) {};
		\node [style=none] (32) at (-6.4, -0.5) {};
		\node [style=none] (33) at (-6.35, -1.5) {};
		\node [style=none] (34) at (2, 0.5) {};
		\node [style=none] (35) at (2, -1.5) {};
		\node [style=none] (36) at (3, 0.5) {};
		\node [style=none] (37) at (3.25, -0.5) {};
		\node [style=none] (38) at (3, -1.5) {};
		\node [style=none] (39) at (2.5, -0.5) {$\task{A}$};
		\node [style=none] (40) at (2, 0) {};
		\node [style=none] (41) at (2, -1) {};
		\node [style=none] (42) at (3.1, 0) {};
		\node [style=none] (43) at (3.15, -1) {};
		\node [style=none] (44) at (4, 0.5) {};
		\node [style=none] (45) at (4, -1.5) {};
		\node [style=none] (46) at (5, 0.5) {};
		\node [style=none] (47) at (5.25, -0.5) {};
		\node [style=none] (48) at (5, -1.5) {};
		\node [style=none] (49) at (4.5, -0.5) {$\task{B}$};
		\node [style=none] (50) at (4, 0) {};
		\node [style=none] (51) at (4, -1) {};
		\node [style=none] (52) at (5.1, 0) {};
		\node [style=none] (53) at (5.15, -1) {};
		\node [style=none] (54) at (-8.5, -0.5) {};
		\node [style=none] (55) at (-8.5, -1.5) {};
		\node [style=none] (56) at (-5.5, -0.5) {};
		\node [style=none] (57) at (-5.5, -1.5) {};
		\node [style=none] (58) at (-8.5, 0) {$\Gamma_\substr{K}$};
		\node [style=none] (59) at (-8.5, -2) {$\Gamma_\substr{D}$};
		\node [style=none] (60) at (-5.5, 0) {$\Gamma_\substr{L}$};
		\node [style=none] (61) at (-5.5, -2) {$\Gamma_\substr{D}$};
		\node [style=none] (62) at (-4.5, -0.5) {};
		\node [style=none] (63) at (-4.5, -1.5) {};
		\node [style=none] (64) at (-1.5, -0.5) {};
		\node [style=none] (65) at (-1.5, -1.5) {};
		\node [style=none] (66) at (-4.5, 0) {$\Gamma_\substr{H}$};
		\node [style=none] (67) at (-1.5, 0) {$\Gamma_\substr{K}$};
		\node [style=none] (68) at (-4.5, -2) {$\Gamma_\substr{C}$};
		\node [style=none] (69) at (-1.5, -2) {$\Gamma_\substr{C}$};
		\node [style=none] (70) at (-5, -1) {$\circ$};
		\node [style=none] (71) at (-0.75, -1) {=};
		\node [style=none] (72) at (1, 0) {};
		\node [style=none] (73) at (6, 0) {};
		\node [style=none] (74) at (2, -1) {};
		\node [style=none] (75) at (1, -1.25) {};
		\node [style=none] (76) at (4.1, -2) {};
		\node [style=none] (77) at (3, -2) {};
		\node [style=none] (78) at (1, -1.75) {};
		\node [style=none] (79) at (6, -1.75) {};
		\node [style=none] (80) at (6, -1.25) {};
		\node [style=none] (81) at (5, -2) {};
		\node [style=none] (83) at (3.5, 0.5) {$\Gamma_\substr{K}$};
		\node [style=none] (85) at (3, -2.5) {$\Gamma_\substr{D}$};
		\node [style=none] (86) at (1.5, -0.75) {$\Gamma_\substr{C}$};
		\node [style=none] (87) at (4, -2.5) {$\Gamma_\substr{C}$};
		\node [style=none] (90) at (1, 1) {};
		\node [style=none] (91) at (1, -3) {};
		\node [style=none] (92) at (6, -3) {};
		\node [style=none] (93) at (6, 1) {};
		\node [style=none] (94) at (0, 0) {};
		\node [style=none] (95) at (7, 0) {};
		\node [style=none] (96) at (0, 0.5) {$\Gamma_\substr{H}$};
		\node [style=none] (97) at (7, 0.5) {$\Gamma_\substr{L}$};
		\node [style=none] (98) at (0, -1.25) {};
		\node [style=none] (99) at (0, -1.75) {};
		\node [style=none] (100) at (7, -1.75) {};
		\node [style=none] (101) at (7, -1.25) {};
		\node [style=none] (102) at (0, -2.25) {$\Gamma_\substr{C} \times \Gamma_\substr{D}$};
		\node [style=none] (103) at (7, -2.25) {$\Gamma_\substr{C} \times \Gamma_\substr{D}$};
		\node [style=none] (104) at (2.1, -2) {};
		\node [style=none] (105) at (6.25, -1) {};
	\end{pgfonlayer}
	\begin{pgfonlayer}{edgelayer}
		\draw (17.center)
			 to (18.center)
			 to (15.center)
			 to (14.center)
			 to (16.center)
			 to cycle;
		\draw (25.center)
			 to [in=270, out=90] (24.center)
			 to (26.center)
			 to (27.center)
			 to (28.center)
			 to cycle;
		\draw (35.center)
			 to (34.center)
			 to (36.center)
			 to (37.center)
			 to (38.center)
			 to cycle;
		\draw (45.center)
			 to [in=270, out=90] (44.center)
			 to (46.center)
			 to (47.center)
			 to (48.center)
			 to cycle;
		\draw (54.center) to (30.center);
		\draw (32.center) to (56.center);
		\draw (33.center) to (57.center);
		\draw (55.center) to (31.center);
		\draw (23.center) to (65.center);
		\draw (22.center) to (64.center);
		\draw (62.center) to (20.center);
		\draw (63.center) to (21.center);
		\draw (42.center) to (50.center);
		\draw (52.center) to (73.center);
		\draw (72.center) to (40.center);
		\draw (98.center)
			 to (75.center)
			 to [in=-180, out=0] (74.center);
		\draw (99.center)
			 to (78.center)
			 to [in=180, out=0] (104.center)
			 to (77.center)
			 to [in=-180, out=0] (51.center);
		\draw (43.center)
			 to [in=180, out=0] (76.center)
			 to (81.center)
			 to [in=-180, out=0] (80.center);
		\draw (53.center)
			 to [in=-180, out=0] (79.center)
			 to (100.center);
		\draw [style=H] (93.center)
			 to (105.center)
			 to (92.center)
			 to (91.center)
			 to (90.center)
			 to cycle;
		\draw (94.center) to (72.center);
		\draw (73.center) to (95.center);
		\draw (80.center) to (101.center);
	\end{pgfonlayer}
\end{tikzpicture}

%% file: figures/partaskcomp.tikz
\begin{tikzpicture}
	\begin{pgfonlayer}{nodelayer}
		\node [style=none] (14) at (-7.5, 0) {};
		\node [style=none] (15) at (-7.5, -2) {};
		\node [style=none] (16) at (-6.5, 0) {};
		\node [style=none] (17) at (-6.25, -1) {};
		\node [style=none] (18) at (-6.5, -2) {};
		\node [style=none] (19) at (-7, -1) {$\task{A}$};
		\node [style=none] (20) at (-7.5, -0.5) {};
		\node [style=none] (21) at (-7.5, -1.5) {};
		\node [style=none] (22) at (-6.4, -0.5) {};
		\node [style=none] (23) at (-6.35, -1.5) {};
		\node [style=none] (24) at (-3.5, 0) {};
		\node [style=none] (25) at (-3.5, -2) {};
		\node [style=none] (26) at (-2.5, 0) {};
		\node [style=none] (27) at (-2.25, -1) {};
		\node [style=none] (28) at (-2.5, -2) {};
		\node [style=none] (29) at (-3, -1) {$\task{B}$};
		\node [style=none] (30) at (-3.5, -0.5) {};
		\node [style=none] (31) at (-3.5, -1.5) {};
		\node [style=none] (32) at (-2.4, -0.5) {};
		\node [style=none] (33) at (-2.35, -1.5) {};
		\node [style=none] (54) at (-4.5, -0.5) {};
		\node [style=none] (55) at (-4.5, -1.5) {};
		\node [style=none] (56) at (-1.5, -0.5) {};
		\node [style=none] (57) at (-1.5, -1.5) {};
		\node [style=none] (58) at (-4.5, 0) {$\Gamma_\substr{L}$};
		\node [style=none] (59) at (-4.5, -2) {$\Gamma_\substr{D}$};
		\node [style=none] (60) at (-1.5, 0) {$\Gamma_\substr{M}$};
		\node [style=none] (61) at (-1.5, -2) {$\Gamma_\substr{D}$};
		\node [style=none] (62) at (-8.5, -0.5) {};
		\node [style=none] (63) at (-8.5, -1.5) {};
		\node [style=none] (64) at (-5.5, -0.5) {};
		\node [style=none] (65) at (-5.5, -1.5) {};
		\node [style=none] (66) at (-8.5, 0) {$\Gamma_\substr{H}$};
		\node [style=none] (67) at (-5.5, 0) {$\Gamma_\substr{K}$};
		\node [style=none] (68) at (-8.5, -2) {$\Gamma_\substr{C}$};
		\node [style=none] (69) at (-5.5, -2) {$\Gamma_\substr{C}$};
		\node [style=none] (70) at (-5, -1) {$\times$};
		\node [style=none] (71) at (-0.75, -1) {=};
		\node [style=none] (72) at (1, 0.25) {};
		\node [style=none] (73) at (6, 0.25) {};
		\node [style=none] (75) at (1, -1.75) {};
		\node [style=none] (78) at (1, -2.25) {};
		\node [style=none] (79) at (6, -2.25) {};
		\node [style=none] (80) at (6, -1.75) {};
		\node [style=none] (90) at (1, 1) {};
		\node [style=none] (91) at (1, -3) {};
		\node [style=none] (92) at (6, -3) {};
		\node [style=none] (93) at (6, 1) {};
		\node [style=none] (94) at (0, 0.25) {};
		\node [style=none] (95) at (7, 0.25) {};
		\node [style=none] (96) at (0, 0.75) {$\Gamma_\substr{H} \times \Gamma_\substr{L}$};
		\node [style=none] (97) at (7, 0.75) {$\Gamma_\substr{K} \times \Gamma_\substr{M}$};
		\node [style=none] (98) at (0, -1.75) {};
		\node [style=none] (99) at (0, -2.25) {};
		\node [style=none] (100) at (7, -2.25) {};
		\node [style=none] (101) at (7, -1.75) {};
		\node [style=none] (102) at (0, -2.75) {$\Gamma_\substr{C} \times \Gamma_\substr{D}$};
		\node [style=none] (103) at (7, -2.75) {$\Gamma_\substr{C} \times \Gamma_\substr{D}$};
		\node [style=none] (104) at (0, -0.25) {};
		\node [style=none] (105) at (1, -0.25) {};
		\node [style=none] (106) at (6, -0.25) {};
		\node [style=none] (107) at (7, -0.25) {};
		\node [style=none] (110) at (3, 0.6) {};
		\node [style=none] (111) at (3, -0.9) {};
		\node [style=none] (112) at (4, 0.6) {};
		\node [style=none] (113) at (4.25, -0.15) {};
		\node [style=none] (114) at (4, -0.9) {};
		\node [style=none] (115) at (3.5, -0.15) {$\task{A}$};
		\node [style=none] (116) at (3, 0.1) {};
		\node [style=none] (117) at (3, -0.4) {};
		\node [style=none] (118) at (4.15, 0.1) {};
		\node [style=none] (119) at (4.15, -0.4) {};
		\node [style=none] (120) at (3, -1.15) {};
		\node [style=none] (121) at (3, -2.65) {};
		\node [style=none] (122) at (4, -1.15) {};
		\node [style=none] (123) at (4.25, -1.9) {};
		\node [style=none] (124) at (4, -2.65) {};
		\node [style=none] (125) at (3.5, -1.9) {$\task{B}$};
		\node [style=none] (126) at (3, -1.65) {};
		\node [style=none] (127) at (3, -2.15) {};
		\node [style=none] (128) at (4.15, -1.65) {};
		\node [style=none] (129) at (4.15, -2.15) {};
		\node [style=none] (130) at (2, -1.65) {};
		\node [style=none] (131) at (2.5, -2.15) {};
		\node [style=none] (132) at (5, -1.65) {};
		\node [style=none] (133) at (4.75, -2.15) {};
		\node [style=none] (138) at (2.5, 0.1) {};
		\node [style=none] (139) at (2, -0.4) {};
		\node [style=none] (140) at (4.75, 0.1) {};
		\node [style=none] (141) at (5, -0.4) {};
		\node [style=none] (142) at (6.25, -1) {};
	\end{pgfonlayer}
	\begin{pgfonlayer}{edgelayer}
		\draw (17.center)
			 to (18.center)
			 to (15.center)
			 to (14.center)
			 to (16.center)
			 to cycle;
		\draw (25.center)
			 to [in=270, out=90] (24.center)
			 to (26.center)
			 to (27.center)
			 to (28.center)
			 to cycle;
		\draw (54.center) to (30.center);
		\draw (32.center) to (56.center);
		\draw (33.center) to (57.center);
		\draw (55.center) to (31.center);
		\draw (23.center) to (65.center);
		\draw (22.center) to (64.center);
		\draw (62.center) to (20.center);
		\draw (63.center) to (21.center);
		\draw (98.center) to (75.center);
		\draw (99.center)
			 to (78.center)
			 to [in=180, out=0] (131.center)
			 to (127.center);
		\draw (129.center)
			 to (133.center)
			 to [in=180, out=0] (79.center)
			 to (100.center);
		\draw [style=H] (92.center)
			 to (91.center)
			 to (90.center)
			 to (93.center)
			 to (142.center)
			 to cycle;
		\draw (94.center)
			 to (72.center)
			 to [in=180, out=0] (138.center)
			 to (116.center);
		\draw (118.center)
			 to (140.center)
			 to [in=-180, out=0] (73.center)
			 to (95.center);
		\draw (80.center) to (101.center);
		\draw (128.center)
			 to (132.center)
			 to [in=-180, out=0, looseness=0.75] (106.center)
			 to (107.center);
		\draw (104.center)
			 to (105.center)
			 to [in=-180, out=0, looseness=0.75] (130.center)
			 to (126.center);
		\draw (113.center) to (114.center);
		\draw (114.center) to (111.center);
		\draw (111.center) to (110.center);
		\draw (110.center) to (112.center);
		\draw (112.center) to (113.center);
		\draw (121.center)
			 to [in=270, out=90] (120.center)
			 to (122.center)
			 to (123.center)
			 to (124.center)
			 to cycle;
		\draw (119.center) to (141.center);
		\draw (139.center) to (117.center);
		\draw [in=-180, out=0, looseness=0.75] (75.center) to (139.center);
		\draw [in=-180, out=0] (141.center) to (80.center);
	\end{pgfonlayer}
\end{tikzpicture}

%% file: figures/seqcompcheck.tikz
\begin{tikzpicture}
	\begin{pgfonlayer}{nodelayer}
		\node [style=none] (0) at (0.25, 0.5) {};
		\node [style=none] (1) at (1, 0) {};
		\node [style=none] (2) at (0.25, -0.5) {};
		\node [style=none] (3) at (0.5, 0) {$U$};
		\node [style=none] (4) at (0.25, 0) {};
		\node [style=none] (18) at (1.5, 0) {$\supseteq$};
		\node [style=none] (20) at (4.25, 0.5) {};
		\node [style=none] (21) at (5, 0) {};
		\node [style=none] (22) at (4.25, -0.5) {};
		\node [style=none] (23) at (4.5, 0) {$T$};
		\node [style=none] (24) at (4.25, 0) {};
		\node [style=none] (26) at (6.75, 0.5) {};
		\node [style=none] (27) at (7.5, 0) {};
		\node [style=none] (28) at (6.75, -0.5) {};
		\node [style=none] (29) at (7, 0) {$S$};
		\node [style=none] (30) at (6.75, 0) {};
		\node [style=none] (31) at (6, 0) {};
		\node [style=none] (37) at (2, 0) {};
		\node [style=none] (38) at (5.5, 0) {$\supseteq$};
		\node [style=none] (41) at (-2.75, 0.5) {};
		\node [style=none] (42) at (-2.75, -0.5) {};
		\node [style=none] (43) at (-1.75, -0.5) {};
		\node [style=none] (44) at (-1.75, 0.5) {};
		\node [style=none] (45) at (-2.25, 0) {$\task{A}$};
		\node [style=none] (46) at (-2.75, 0) {};
		\node [style=none] (47) at (-3.5, 0) {};
		\node [style=none] (48) at (-1.65, 0) {};
		\node [style=none] (49) at (-1.25, 0.5) {};
		\node [style=none] (50) at (-1.25, -0.5) {};
		\node [style=none] (51) at (-0.25, -0.5) {};
		\node [style=none] (52) at (-0.25, 0.5) {};
		\node [style=none] (53) at (-0.75, 0) {$\task{B}$};
		\node [style=none] (54) at (-1.25, 0) {};
		\node [style=none] (55) at (-0.15, 0) {};
		\node [style=none] (56) at (2.75, 0.5) {};
		\node [style=none] (57) at (2.75, -0.5) {};
		\node [style=none] (58) at (3.75, -0.5) {};
		\node [style=none] (59) at (3.75, 0.5) {};
		\node [style=none] (60) at (3.25, 0) {$\task{A}$};
		\node [style=none] (61) at (2.75, 0) {};
		\node [style=none] (62) at (3.85, 0) {};
	\end{pgfonlayer}
	\begin{pgfonlayer}{edgelayer}
		\draw (2.center)
			 to (0.center)
			 to (1.center)
			 to cycle;
		\draw (22.center)
			 to (20.center)
			 to (21.center)
			 to cycle;
		\draw (27.center)
			 to (28.center)
			 to (26.center)
			 to cycle;
		\draw (30.center) to (31.center);
		\draw (41.center)
			 to (44.center)
			 to (48.center)
			 to (43.center)
			 to (42.center)
			 to cycle;
		\draw (47.center) to (46.center);
		\draw (49.center)
			 to (52.center)
			 to (55.center)
			 to (51.center)
			 to (50.center)
			 to cycle;
		\draw (56.center)
			 to (59.center)
			 to (62.center)
			 to (58.center)
			 to (57.center)
			 to cycle;
		\draw (48.center) to (54.center);
		\draw (55.center) to (4.center);
		\draw (37.center) to (61.center);
		\draw (62.center) to (24.center);
	\end{pgfonlayer}
\end{tikzpicture}

%% file: figures/parcompcheck.tikz
\begin{tikzpicture}
	\begin{pgfonlayer}{nodelayer}
		\node [style=none] (20) at (4.25, 0.5) {};
		\node [style=none] (21) at (5, 0) {};
		\node [style=none] (22) at (4.25, -0.5) {};
		\node [style=none] (23) at (4.5, 0) {$U$};
		\node [style=none] (24) at (4.25, 0) {};
		\node [style=none] (26) at (7, 0.5) {};
		\node [style=none] (27) at (7.75, 0) {};
		\node [style=none] (28) at (7, -0.5) {};
		\node [style=none] (29) at (7.25, 0) {$S$};
		\node [style=none] (30) at (7, 0) {};
		\node [style=none] (31) at (6, 0) {};
		\node [style=none] (37) at (1.25, 0) {};
		\node [style=none] (38) at (5.5, -0.75) {$\supseteq$};
		\node [style=none] (48) at (4.25, -1) {};
		\node [style=none] (49) at (5, -1.5) {};
		\node [style=none] (50) at (4.25, -2) {};
		\node [style=none] (51) at (4.5, -1.5) {$V$};
		\node [style=none] (52) at (4.25, -1.5) {};
		\node [style=none] (59) at (1.25, -1.5) {};
		\node [style=none] (61) at (7, -1) {};
		\node [style=none] (62) at (7.75, -1.5) {};
		\node [style=none] (63) at (7, -2) {};
		\node [style=none] (64) at (7.25, -1.5) {$T$};
		\node [style=none] (65) at (7, -1.5) {};
		\node [style=none] (66) at (6, -1.5) {};
		\node [style=none] (67) at (1.5, 0.25) {$\bar{X}$};
		\node [style=none] (68) at (1.5, -1.75) {$\bar{Z}$};
		\node [style=none] (69) at (3.75, 0.25) {$\bar{Y}$};
		\node [style=none] (70) at (3.75, -1.75) {$\bar{W}$};
		\node [style=none] (71) at (6.25, 0.25) {$\bar{X}$};
		\node [style=none] (72) at (6.25, -1.75) {$\bar{Z}$};
		\node [style=none] (73) at (2.25, 0.5) {};
		\node [style=none] (74) at (2.25, -0.5) {};
		\node [style=none] (75) at (3.25, -0.5) {};
		\node [style=none] (76) at (3.25, 0.5) {};
		\node [style=none] (77) at (2.75, 0) {$\task{A}$};
		\node [style=none] (78) at (2.25, 0) {};
		\node [style=none] (79) at (3.35, 0) {};
		\node [style=none] (80) at (2.25, -1) {};
		\node [style=none] (81) at (2.25, -2) {};
		\node [style=none] (82) at (3.25, -2) {};
		\node [style=none] (83) at (3.25, -1) {};
		\node [style=none] (84) at (2.75, -1.5) {$\task{B}$};
		\node [style=none] (85) at (2.25, -1.5) {};
		\node [style=none] (86) at (3.35, -1.5) {};
	\end{pgfonlayer}
	\begin{pgfonlayer}{edgelayer}
		\draw (22.center)
			 to (20.center)
			 to (21.center)
			 to cycle;
		\draw (27.center)
			 to (28.center)
			 to (26.center)
			 to cycle;
		\draw (30.center) to (31.center);
		\draw (50.center)
			 to (48.center)
			 to (49.center)
			 to cycle;
		\draw (61.center)
			 to (62.center)
			 to (63.center)
			 to cycle;
		\draw (65.center) to (66.center);
		\draw (73.center)
			 to (76.center)
			 to (79.center)
			 to (75.center)
			 to (74.center)
			 to cycle;
		\draw (80.center)
			 to (83.center)
			 to (86.center)
			 to (82.center)
			 to (81.center)
			 to cycle;
		\draw (59.center) to (85.center);
		\draw (86.center) to (52.center);
		\draw (37.center) to (78.center);
		\draw (79.center) to (24.center);
	\end{pgfonlayer}
\end{tikzpicture}